\documentclass[12pt,a4paper]{article}
\usepackage{latexsym,amssymb,amsmath,amsthm,amsfonts, amscd, enumerate,verbatim,xspace,exscale}
\usepackage{geometry}
\usepackage{authblk}

\numberwithin{equation}{section}

\input xy
\xyoption{all} \CompileMatrices \UseComputerModernTips

\usepackage{graphicx,tabularx}
\usepackage{hyperref}
\usepackage{tikz}
\usepackage{amsmath,amsbsy,amsfonts,amssymb}
\usepackage{xcolor}
\usepackage{amsthm}
\usepackage{mathtools}
\usepackage{dsfont}
\renewcommand{\theequation}{\arabic{section}.\arabic{equation}}

\theoremstyle{plain}
\newtheorem{theorem}{Theorem}[section]
\newtheorem{lemma}[theorem]{Lemma}
\newtheorem{proposition}[theorem]{Proposition}
\newtheorem{corollary}[theorem]{Corollary}
\newtheorem{remark}[theorem]{Remark}

\theoremstyle{definition}
\newtheorem{definition}[theorem]{Definition}

\renewcommand{\theequation}{\thesection.\arabic{equation}}

\def\R{\mathbb{R}}

\date{}
\begin{document}

\title{Hierarchical null controllability of a degenerate parabolic equation with nonlocal coefficient}

\author[1]{Juan L\'imaco \thanks{email: jlimaco@id.uff.br - Corresponding author} }
\author[1]{João Carlos Barreira \thanks{email: jcbarreira95@gmail.com}}
\author[1]{Suerlan Silva \thanks{email: suerlansilva@id.uff.br}}
\author[1]{Luis P. Yapu
\thanks{email: luis.yapu@gmail.com}}

\affil[1]{Instituto de Matem\'atica e Estat\'\i stica, Universidade Federal Fluminense, Rua Prof. Marcos Waldemar de Freitas Reis S/N (Campus do Gragoat\'a), Niter\'oi, CEP 24210-201, Rio de Janeiro, Brazil}


\maketitle

\begin{abstract}
\noindent In this paper we use a Stackelberg-Nash strategy to show the local null controllability of a parabolic equation where the diffusion coefficient is the product of a degenerate function in space and a nonlocal term. We consider one control called \textit{leader} and two controls called \textit{followers}. To each leader we associate a Nash equilibrium corresponding to a bi-objective optimal control problem; then, we find a leader that solves the null controllability problem. The linearized degenerated system is treated adapting Carleman estimates for degenerated systems from Demarque, Límaco and Viana \cite{DemarqueLimacoViana_deg_sys2020} and the local controllability of the non-linear system is obtained using Liusternik's inverse function theorem. The nonlocal coefficient originates a multiplicative coupling in the optimality system that gives rise to interesting calculations in the applications of the inverse function theorem.  
\end{abstract}

\textbf{MSC Classification (2020)}: Primary: 35K65, 93B05; Secondary: 93C10. 

\textbf{keywords}: Degenerate parabolic equations, Controllability, Nonlinear systems in Control Theory, nonlocal term, Carleman inequalities.


\section{Introduction} \label{S:Intro}

Let $T>0$, and define the cylinder $Q=(0,1)\times (0,T)$. The aim of this paper is to study a multi-objective problem based on the nonlocal degenerate parabolic equation

\begin{equation}\label{eq:PDE}
	\left\{\begin{aligned}
		&y_t - \left(a(x)\ell\left(\int_{0}^{1}y \ dx\right) y_x\right)_x = h \mathds{1}_{\mathcal{O}} + v^{1} \mathds{1}_{\mathcal{O}_{1}} + v^{2} \mathds{1}_{\mathcal{O}_2} &&\text{in}&& Q,\\ &y(0,t)=y(1,t)=0&&\text{on}&& (0,T), \\ 
		&y(\cdot,0) = y_0 &&\text{in}&& (0,1),
	\end{aligned}
	\right.
\end{equation}
where $y=y(x,t)$ is the state, $y_{0}$ is the initial data, $\mathds{1}_A$ denotes the characteristic function of the set $A$, and $h,v^{1},v^{2}$ are control functions acting on the sets $\mathcal{O}, \mathcal{O}_1,\mathcal{O}_2\subset (0,1)$ respectively. We suppose that the function $\ell:\mathbb{R} \rightarrow \mathbb{R}$ is of class $C^2$ with bounded derivative and $\ell(0)>0$. Let us assume the function $a$ satisfies the \emph{weakly degenerate condition}\footnote{Some comments about the strong case will give in Section \ref{sec:final_remarks}.} defined in \cite{Alabau_cannarsa_fragnelli-06} and \cite{de2023null}, i.e, $a\in C([0,1]) \cap C^1((0,1]),$ 
\begin{equation}\label{conditionunder-a(x)}
a(x) > 0,\ \ a(0)=0,\ \ \text{and}\ \ a'(x) \ge 0 \ \ \text{in} \ \ (0,1]
\end{equation}
and
\begin{equation}\label{K-estimate-a(x)}
xa'(x)\leq K a(x), \ \ \text{with}\ \ K\in [0,1).
\end{equation}
In other words, the function $a$ exhibits a behavior similar to $x^\gamma$, with $\gamma \in (0,1)$. Due to the degeneracy of the equation, the initial condition $y_0$  is considered in the weighted space $H_a^1(0,1)$, which will be formally defined in Section~\ref{Sec:characterization}.

The nonlocal terms in \eqref{eq:PDE} naturally arise in various physical models. For example, they can appear in heat conduction in materials with memory, nuclear reactor dynamics, and population dynamics. In the latter case, such a system can model the dispersion of a gene within a given population. In this context, $x$ represents the gene type, $y = y(t, x)$ denotes the distribution of individuals at time $t$ with gene type $x$ and $a(x)$ is the gene dispersion coefficient. From a genetic perspective, this degeneracy property is natural: if a population does not contain a certain gene type, it cannot be transmitted to its offspring; see for instance \cite{echarroudi2014null}. For additional references on works involving nonlocal terms, see \cite{chipot2003remarks,zheng2005asymptotic,shimakura1992partial}. 

The control problem studied in this work involves achieving three distinct objectives, each associated with a specific control function: $h$, $v_1$, and $v_2$. The controls $v_1$ and $v_2$ aim to provide the best approximation of the solution $y$ of \eqref{eq:PDE} (associated with $h$) for the prescribed target configurations $y_{i,d}$ within specified regions $\mathcal{O}_{i,d}$, for $i = 1, 2$. The control $h$, in turn, is designed to drive the solution $y$ of \eqref{eq:PDE}, associated with the pair $(v^1, v^2)$, to zero. These three objectives are inherently conflicting and therefore must be addressed simultaneously. To this end, we adopt a combined strategy. Specifically, we utilize a noncooperative Nash equilibrium approach \cite{EssaysonGameTheory} for the relationship between the minimizing controls $v^{1}$ and $v^{2}$, coupled with a hierarchical Stackelberg strategy \cite{von1934marktform} where $h$ acts as the leader and $v^{1},v^{2}$ serve as followers. This integrated approach, known as the Stackelberg-Nash strategy, provides an effective framework for addressing multi-objective control problems. We now proceed to describe our problem formulation in precise terms.

Let us consider the sets $\mathcal{O}_{1,d}, \mathcal{O}_{2,d}\subset (0,1)$ representing the observation domains of the followers, respectively. Given functions $y_{i,d} = y_{i,d}(x,t)$ and constants $\alpha_i,\mu_i>0$, we define the cost functionals
\begin{equation}\label{eq:def_Ji}
J_{i}(h,v^{1},v^{2}) = \frac{\alpha_i}{2} \int_{\mathcal{O}_{i,d}\times(0,T)} |y-y_{i,d}|^2 dxdt + \frac{\mu_i}{2} \int_{\mathcal{O}_i\times(0,T)} |v^i|^2 dxdt.
\end{equation}
The concepts of Nash Equilibrium and null controllability are given below:
\begin{definition} Given $y_{0}\in H_a^1(0,1)$ and for a fixed $h\in L^{2}(\mathcal{O}\times(0,T))$, a pair $(v^{1},v^{2})\in L^{2}(\mathcal{O}_{1}\times(0,T))\times L^{2}(\mathcal{O}_{2}\times(0,T))$  is a \emph{Nash equilibrium} for the functionals $J_1$ and $J_2$ if
\begin{equation}\label{minimousfunctional}
\left\{\begin{aligned}&J_1(h,v^{1},v^{2}) = \min_{\hat{v}^{1}\in L^{2}(\mathcal{O}_{1}\times(0,T))} J_1(h,\hat{v}^{1}, v^{2}), \\
&J_2(h,v^{1},v^{2}) = \min_{\hat{v}^{2}\in L^{2}(\mathcal{O}_{2}\times(0,T))} J_1(h,v^{1},\hat{v}^{2}).
\end{aligned}
\right.
\end{equation}    
\end{definition}
Note that, if the functionals $J_{1}$ and $J_{2}$  are convex, then a pair $(v_{1},v_{2})$ is a Nash equilibrium if and only if
\begin{equation*}
\left\{\begin{aligned}
&J_1'(h,v^{1},v^{2})\cdot(\hat{v}^{1},0) = 0, \quad \forall \hat{v}^{1} \in L^{2}(\mathcal{O}_{1}\times(0,T)), \\
&J_2'(h,v^{1},v^{2})\cdot(0,\hat{v}^{2}) = 0, \quad \forall \hat{v}^{2} \in L^{2}(\mathcal{O}_{2}\times(0,T)).
\end{aligned}
\right.
\end{equation*}
In nonlinear systems one expects to lose convexity of the functional, and the Nash equilibrium condition \eqref{minimousfunctional} is not necessarily equivalent to affirmation above. For this reason, it is convenient to weaken the definition of equilibrium.  
\begin{definition} Given $y_{0}\in H_a^1(0,1)$ and for a fixed $h\in L^{2}(\mathcal{O}\times(0,T))$, a pair $(v^{1},v^{2})\in L^{2}(\mathcal{O}_{1}\times(0,T))\times L^{2}(\mathcal{O}_{2}\times(0,T))$  is a \emph{Nash quasi-equilibrium} for the functionals $J_1$ and $J_2$ if
\begin{equation}\label{derivateoffunctional}
\left\{\begin{aligned}
&J_1'(h,v^{1},v^{2})\cdot(\hat{v}^{1},0) = 0, \quad \forall \hat{v}^{1} \in L^{2}(\mathcal{O}_{1}\times(0,T)), \\
&J_2'(h,v^{1},v^{2})\cdot(0,\hat{v}^{2}) = 0, \quad \forall \hat{v}^{2} \in L^{2}(\mathcal{O}_{2}\times(0,T)).
\end{aligned}
\right.
\end{equation}
\end{definition}
Once the Nash equilibrium have been identified for any $h$, for any $T>0$ we look for a control $h \in L^2(\mathcal{O}\times(0,T))$ such that the system is \emph{null controllable} at time $T$, i.e. 
\begin{equation*}
y(x,T) = 0, \text{ for } x \in (0,1).    
\end{equation*}

This approach was introduced in \cite{Lions2,Lions1}, where the hierarchical Stackelberg strategy was developed for parabolic and hyperbolic problems, respectively. Subsequently, \cite{Di-L} established a connection between the Stackelberg strategy and the Nash equilibrium \cite{EssaysonGameTheory}, adapting it to scenarios with one leader and multiple followers. In this framework, the main objective is to simultaneously minimize $N$ cost functionals to achieve a Nash equilibrium among the followers, while the leader drives the system to satisfy an approximate controllability requirement -- thus implementing  the Stackelberg–Nash strategy.

More recently, \cite{F-E-M} investigated Stackelberg–Nash strategies and established the first hierarchical results in the context of exact controllability for a class of parabolic equations, both linear and semilinear. They also identified conditions under which a Nash quasi-equilibrium is, in fact, a Nash equilibrium. These results were later refined in \cite{F-E-M2} by relaxing the requirements on the followers’ observation domains. Continuing in this direction, \cite{Huaman2022} extended the application of Stackelberg–Nash strategies to parabolic equations with a nonlinear diffusion term. Further developments include \cite{araruna2018stackelberg}, which provides the first results for degenerate parabolic equations using the same strategies, and \cite{djomegne2022hierarchical}, which presents a comprehensive study of hierarchical control for the semilinear case. In \cite{Teresa2024}, the author considered a hierarchy of two controls applied to the Boussinesq system, where the objective of the \emph{leader} control is to achieve null controllability, while the objective of the \emph{follower} control is to keep the state close to a prescribed trajectory. 

Modifying the hierarchical strategy previously discussed, \cite{calsavara2022new} investigated optimal controllability and controllability within a hierarchical control framework, one leader and one follower, for linear and semilinear parabolic PDEs. In this setting, the follower was assigned the responsibility for ensuring the system's controllability, while the choice of the leader was defined as the solution to an optimal control problem, see Subsection \ref{Subsection strong degenerate} for more details. According to the authors, the motivation for this approach was drawn from applications in environmental sciences, particularly in situations where a government seeks to keep pollution levels within acceptable limits, while followers pursue their own objectives.

Now, while most previous studies addressed internal control problems, \cite{araruna2020hierarchical} examined scenarios involving boundary controls, analyzing a mixed control framework that combines both internal and boundary controls. The case in which all controls are applied on the boundary was subsequently solved in \cite{djomegne2025hierarchical}. Regarding dynamic boundary conditions, \cite{IdrissBoutaayamou} initiated studies in this direction, and \cite{oukdach2025multi} considered the stochastic parabolic case. Concerning stochastic equations, it is worth highlighting the first contribution in this context \cite{zhang2019controllability}, which was later extended to the degenerate case in \cite{yu2023two}. 

As established in earlier works, degenerate systems have garnered significant attention following the foundational results presented in \cite{Alabau_cannarsa_fragnelli-06} and \cite{Cannarsa_Teresa-09}. Within this framework, notable contributions include the coupled system studied in \cite{Khodja_degen-11}, with subsequent improvements in \cite{Khodja_coupled_deg-13}. Motivated by climate change applications, \cite{FLORIDIA-14} pioneered bilinear controllability analysis for these systems. Addressing similar environmental concerns, \cite{akil2025non} proposed a more realistic model incorporating Robin boundary conditions. Regarding nonlocal coefficients, significant advances were achieved in \cite{DemarqueLimacoViana_deg_eq2018} and \cite{DemarqueLimacoViana_deg_sys2020}, building upon the theoretical framework established in \cite{EFC_Lim_Silvano_nonlocal-12}.

Our main result establishes the existence of a Nash quasi-equilibrium pair of controls $(v^{1},v^{2})$ for $J_1$ and $J_2$, leading to the local null controllability of equation \eqref{eq:PDE}. 

\begin{theorem}\label{thm:local_null_controllability}
Assume the hypotheses of the setting in \eqref{eq:PDE} and that, 
\begin{equation*}
\mathcal{O}_d := \mathcal{O}_{1,d} = \mathcal{O}_{2,d}
\end{equation*}
and
\begin{equation*}
\mathcal{O}_{d} \cap \mathcal{O} \neq \emptyset.
\end{equation*}

Then, for any $T>0$, there exist $\delta>0$ and a positive function $\rho_2(t)$ blowing up at $t=T$ such that, if $\rho_2\, y_{i,d} \in L^2(\mathcal{O}_d \times (0,T))$ and $y_0 \in H_a^1(0,1)$ satisfies
\begin{equation*}
\|y_{0}\|_{H_a^1(0,1)} \leq \delta,   
\end{equation*}
then there exist a control $h \in L^2(\mathcal{O}\times(0,T))$ and an associated Nash quasi-equilibrium $(v^{1},v^{2})$ such that the solution of \eqref{eq:PDE} is null-controllable at time $T$.
\end{theorem}
Our next result establishes sufficient conditions for the Nash quasi-equilibrium pair $(v^{1},v^{2})$ obtained in Theorem~\ref{thm:local_null_controllability} to be a Nash equilibrium, from which we deduce the following result.
\begin{proposition} 
\label{thm:nash_equilibrium}
Let $y_0 \in H_a^1(0,1)$.
Under the hypotheses of Theorem \ref{thm:local_null_controllability}, if $\mu_1,\mu_2$ are sufficiently large, then the Nash quasi-equilibrium pair $(v^{1},v^{2})$ obtained in Theorem \ref{thm:local_null_controllability} is a Nash equilibrium for \eqref{eq:def_Ji} associated to equation \eqref{eq:PDE}.
\end{proposition}
\color{black}

\noindent{\bf Outline of the paper}.
In Section \ref{Sec:characterization} we characterize Nash quasi-equilibrium and we prove Proposition \ref{thm:nash_equilibrium} implying the convexity of the functionals and the fact that the Nash quasi-equilibrium are Nash equilibrium. In Section \ref{Sec:null_linearized_system} we show the null controllability of the linearized system adapting a Carleman estimate for a system of degenerated parabolic equations from \cite{DemarqueLimacoViana_deg_sys2020}. The main result of this section is Theorem \ref{theorem case linear} which besides the null controllability establishes additional estimates needed in the controllability of the nonlinear system. In Section \ref{control for nonlinear system} we apply Liusternik inverse function theorem to show the local null controllability of \eqref{eq:PDE}. The proof of the main result, Theorem \ref{thm:local_null_controllability}, is a consequence of Lemmas \ref{A bem definido}--\ref{Mapa sobrejetivo} where we verify the hypothesis of Liusternik's theorem and apply it to a suitable map whose local surjectivity implies the local controllability result.
In Section \ref{sec:final_remarks}, first we explain an application to a radial multidimensional equation which in polar coordinates can be rewritten in the form of equation  \eqref{eq:PDE}, and we present some related comments and open issues. Finally in Appendix \ref{appendix A} we show two technical results needed in Section \ref{control for nonlinear system}, in the Appendix \ref{appendix B} we show the technical estimates for continuity map and in Appendix \ref{appendix C} we show the well-posedness of the optimality system obtained in Section \ref{Sec:characterization}.

\section{Existence and uniqueness of Nash equilibrium}
\label{Sec:characterization}

In this section, we prove Proposition \ref{thm:nash_equilibrium}. Specifically, for each control $h \in L^2(\mathcal{O} \times (0,T))$, we establish the existence and uniqueness of a Nash equilibrium pair $(v^{1}, v^{2})$ associated with the functionals $J_1$ and $J_2$. This result, in particular, implies the convexity of the functionals and guarantees that $(v^1, v^2)$ corresponds to the solution of a minimization problem, which is characterized by the conditions \eqref{minimousfunctional} and \eqref{derivateoffunctional}. We begin by characterizing the equilibrium pair that satisfies \eqref{derivateoffunctional} through an optimality system. Then, after proving the well-posedness of this system, we verify that the obtained pair also satisfies the minimization condition \eqref{minimousfunctional}, thus confirming it as the unique Nash equilibrium.

To develop this demonstration, we define below the weighted spaces associated with the function $a$, originally introduced and studied in the context of controllability in \cite{Alabau_cannarsa_fragnelli-06}.
 \begin{eqnarray*} 
H_{a}^{1}(0,1)&=&\left\{ u \in L^2(0,1); u \text{ is absolutely continuous in } (0,1],\right.\\ 
 &\qquad&\left.\sqrt{a} u_x \in L^2(0,1) \text{ and } u(0)=u(1)=0\right\},
\end{eqnarray*}
with the norm $\| u \|_{H_{a}^{1}(0,1)}^2 = \|u\|_{L^2(0,1)}^2 + \|\sqrt{a} u_x \|_{L^2(0,1)}^2,$
and
\begin{eqnarray*}
H_a^2(0,1) &=& \{ u \in H_a^1(0,1) \ ; \ a u_x \in H^1(0,1) \},
\end{eqnarray*}
with the norm
$
\| u \|_{H_a^2(0,1)}^2 = \|u\|_{H_a^1(0,1)}^2 + \|(a u_x)_x \|_{L^2(0,1)}^2.
$

These spaces will be employed throughout the remainder of this work.

\subsection{Characterization of Nash quasi-equilibrium}
Let us introduce the spaces $\mathcal{V}_{i}=L^{2}(\mathcal{O}_{i}\times(0,T))$ and $\mathcal{V}=\mathcal{V}_{1}\times\mathcal{V}_{2}$.
Assume that $(v^{1},v^{2})$ is a Nash quasi-equilibrium. We will compute the expressions
\begin{equation}
\left\{
    \begin{array}{llll}
    J_1'(h; v^{1},v^{2})\cdot (\hat{v}^{1},0)=0,\quad \forall \hat{v}_1\in \mathcal{V}_{1},
    \\      
    J_2'(h; v^{1},v^{2})\cdot (0,\hat{v}^{2})=0,\quad \forall \hat{v}_2\in \mathcal{V}_{2}.
    \end{array}
    \right.
\end{equation}
Given a $\hat{v}_{1} \in \mathcal{V}_{1}$ and  $\varepsilon>0$, consider a small perturbation of \eqref{eq:def_Ji} around $v_{1}$.  We have \color{black}

$$
J_1(h,v^{1}+\varepsilon\hat{v}^{1},v^{2}) = \frac{\alpha_1}{2} \int_{\mathcal{O}_{1,d}\times(0,T)} |y^{\varepsilon}-y_{1,d}|^2 dxdt + \frac{\mu_1}{2} \int_{\mathcal{O}_1\times(0,T)} | v^{1} + \varepsilon \hat{v}^{1} |^2 dxdt,
$$
where $y^\varepsilon$ is the solution of
\begin{equation}\label{eq:PDE_lambda_1}
\left\{\begin{aligned}
    &y_t^\varepsilon - \left(a(x)\ell\left(\int_0^1 y^\varepsilon dx\right) y_x^\varepsilon\right)_x = h \mathds{1}_{\mathcal{O}} + (v^{1} + \varepsilon \hat{v}^{1}) \mathds{1}_{\mathcal{O}_1} + v^{2} \mathds{1}_{O_2}, &&\text{in}&& Q, \\
&y^\varepsilon(0,t)=y^\varepsilon(1,t)=0, &&\text{on}&&\ (0,T), \\
&y^\varepsilon(\cdot,0) = y_{0}, &&\text{in}&&\ (0,1).
\end{aligned}
	\right.
\end{equation}
Dividing \eqref{eq:PDE} from \eqref{eq:PDE_lambda_1} and divide by $\varepsilon$, we find that
\begin{equation*}
\frac{y_t^\varepsilon - y_t}{\varepsilon} + \left( a(x) \left[ \ell\left(\int_0^1 y^\varepsilon dx\right) \frac{y_x^\varepsilon -y_x}{\varepsilon}  + \frac{1}{\varepsilon}\left(\ell\left(\int_0^1 y^\varepsilon dx\right) - \ell\left(\int_0^1 y dx\right)\right) y_x \right] \right)_x =  \hat{v}^{1} \mathds{1}_{\mathcal{O}_1}.
\end{equation*}
Using the continuity of the initial data of \eqref{eq:PDE}, we have $y^{\varepsilon} \to y$ as $\varepsilon \to 0$. Let us define the limit
\begin{equation}
    \omega^{1} = \lim_{\varepsilon \to 0} \frac{y^{\varepsilon} - y}{\varepsilon}.
\end{equation}
In particular, we obtain
\begin{equation}\label{eq:PDE_w1}
\left\{\begin{array}{lll}
\omega_{t}^{1} - \left( a(x) \ell\left(\displaystyle\int_0^1 y \,dx\right) \omega_{x}^{1} \right)_x\\ - \left( a(x) \ell'\left(\displaystyle\int_0^1 y\, dx\right) \displaystyle\int_0^1 \omega^1(z,t) dz \ y_x \right)_x = \hat v^{1}  \mathds{1}_{\mathcal{O}_1} &\text{in}& Q, \\ 
\omega^1 (0,t)=\omega^1 (1,t) = 0 &\text{on}& (0,T), \\
\omega^1(\cdot,0) = 0  &\text{in}& (0,1).
\end{array}
	\right.
\end{equation} 
The adjoint system of \eqref{eq:PDE_w1} is
\begin{equation}\label{eq:adjoint_of_w1}
\left\{\begin{aligned}
&-p_t^1 - \left( a(x) \ell\left(\displaystyle\int_0^1 y \,dx\right) p_{x}^{1} \right)_{x}\\ &+\ell'\left(\int_0^1 y \,dx\right)\int_0^1 ( a(x) y_x \ p^1_x) dx = \alpha_1 (y-y_{1,d}) \mathds{1}_{\mathcal{O}_{1,d}} &&\text{in}&& Q, \\
&p^1(0,t)=p^1 (1,t) = 0 &&\text{on}&& (0,T), \\
&p^1(\cdot,T) = 0 &&\text{in}&& (0,1).
\end{aligned}
\right.
\end{equation} 
Multiplying the first equation of \eqref{eq:adjoint_of_w1} by $\omega^1$ and integrating in Q we get
\begin{equation*}
\int_Q \hat{v}^{1} \mathds{1}_{\mathcal{O}_1} p^1 \ dxdt = \int_Q \alpha_1 (y-y_{1,d}) \mathds{1}_{\mathcal{O}_1}\omega^1 \ dxdt.
\end{equation*}
We obtained
\begin{equation}\label{derivate_functional_1_test}
J'_{1}(h,v^{1},v^{2})\cdot(\hat{v}^{1},0) = \alpha_1 \int_{\mathcal{O}_{1,d}\times(0,T)} (y-y_{1,d}) \omega^1 \ dxdt + \mu_1 \int_{\mathcal{O}_1\times (0,T)}  v^{1} \hat{v}^{1} \ dxdt = 0,
\end{equation}
and comparing the last two expression we get
\begin{equation*}
\int_Q \hat{v}^{1} \mathds{1}_{\mathcal{O}_{1}} p^1 dxdt = - \mu_1 \int_{\mathcal{O}_{1} \times (0,T)} \hat{v}^{1} v^{1} \ dxdt,
\end{equation*}
from which
\begin{equation*}
v^{1} = -\frac{1}{\mu_1} p^1 \mathds{1}_{\mathcal{O}_{1}}.
\end{equation*}
Analogously,  we can obtain 
\begin{equation*}
v^{2} = -\frac{1}{\mu_2} p^2 \mathds{1}_{\mathcal{O}_{2}}.
\end{equation*}
Then, we can formulate the optimality system:
\begin{equation}\label{eq:optimality_system}
\left\{\begin{aligned}
&y_t - \left(a(x)\ell\left(\int_0^1 y\, dx\right) y_x\right)_x = h \mathds{1}_{\mathcal{O}} - \frac{1}{\mu_1} p^1 \mathds{1}_{\mathcal{O}_{1}} - \frac{1}{\mu_2} p^2 \mathds{1}_{\mathcal{O}_{2}} &&\text{in}&& Q, \\
&-p_t^i - \left( a(x) \ell\left(\int_0^1 y\, dx\right) p_x^i \right)_x \\ 
&+\ell'\left(\int_0^1 y \,dx\right) \int_0^1 ( a(x) y_x \ p_x^i) dx = \alpha_i (y-y_{i,d}) \mathds{1}_{\mathcal{O}_{i,d}}  &&\text{in}&& Q,\\
&y(0,t)=y(1,t)=0, \ p^i(0,t) = p^i(1,t) = 0 &&\text{on}&& (0,T), \\
&y(\cdot,0) = y_{0}, \ \ p^i(\cdot,T) = 0  &&\text{in }&& (0,1). \\
\end{aligned}
\right.
\end{equation}
The existence and uniqueness of the system \eqref{eq:optimality_system} is obtained in the following proposition.

\begin{proposition}
Let $y_0 \in H^1_a(0,1)$,  $h \in L^2(0,T;L^2(\mathcal{O}))$ and  $y_{i,d} \in L^2(0,T;L^2(\mathcal{O}_{i,d}))$, with $i=\lbrace 1,2\rbrace$. Then, there exists a constant $\sigma > 0$ such that, if
$$
\|y_0\|_{H^1_a(0,1)} 
+ \|h\|_{L^2(0,T;L^2(\mathcal{O}))} + \sum_{i=1}^2 \|y_{i,d}\|_{L^2(0,T;L^2(\mathcal{O}_{i,d}))} \leq \sigma,
$$
the system \eqref{eq:optimality_system} has a unique solution $(y,p^1,p^2)$ satisfying
\begin{equation*}
\begin{split}
\|(y,p^1,p^2)\|_{L^2(0,T;L^2(0,1)) \cap H^1(0,T;H^2_a(0,1))} \leq C &\left(\|y_0\|_{H^1_a(0,1)} + \|h\|_{L^2(0,T;L^2(\mathcal{O}))} + \sum_{i=1}^2 \|y_{i,d}\|_{L^2(0,T;L^2(\mathcal{O}_{i,d}))} \right),
\end{split}    
\end{equation*}
where $C$ is a positive constant independent of $\mu_1$ and $\mu_2$. 
\end{proposition}
The proof of this result is given in Appendix \ref{appendix C}.

\subsection{Equilibrium and quasi-equilibrium}
As already mentioned, in our case the convexity of the functionals $J_{i}$ are not guaranteed. Consequently, the existing definition of Nash equilibrium may not be appropriate for linear scenarios. A re-evaluation and redefinition of Nash optimality are therefore warranted.
\label{subsec:proof_prop_nash_eq}
\begin{proof}[Proof of Proposition \ref{thm:nash_equilibrium}]
We will prove that the pair $(v^{1}, v^{2})$ is a Nash equilibrium. For this, we will prove that this pair satisfies \eqref{minimousfunctional}. That is, having already obtained the critical point $(v^{1}, v^{2})$ for the functionals $J_1$ and $J_2$, we will now use the second derivative test to prove that it is a minimum point for $J_1$ and $J_2$.

Since $(v^{1},v^{2})\in\mathcal{V}$ is a Nash quasi-equilibrium, we have 
\begin{equation}\label{nas-def.q.e}
\left\{\begin{aligned}
&\alpha_i \int_{\mathcal{O}_{i,d}\times(0,T)} (y-y_{i,d})\omega^i \ dxdt + \mu_i \int_{\mathcal{O}_{i} \times (0,T)} \hat v^i v^i \ dxdt= 0\\
&\forall\ \hat{v}^{i}\in \mathcal{V}_{i},\ \ \ v^{i}\in \mathcal{V}_{i}, \ \ \ i=\lbrace1,2\rbrace,
\end{aligned}
\right.
\end{equation}
where $\omega^i$ satisfies
\begin{equation}\label{eqomega}
		\left\{\begin{aligned}
			&\omega_t^i - \left( a(x) \left[ \ell\left(\int_0^1 y\, dx\right) \omega_x^i   - \ell'\left(\int_0^1 y\, dx\right) \int_0^1 \omega^i \, dx \ y_{x} \right] \right)_x=\hat v^1 \mathds{1}_{\mathcal{O}_{i}} &&\text{in}&& Q,\\
			&\omega^{i}(0,t)=\omega^{i}(1,t)=0 &&\text{on} && (0,T), \\
			&\omega^{i}(\cdot,0)=0 &&\text{in} && (0,1).
		\end{aligned}
		\right.
\end{equation}
Given $\varepsilon\in\mathbb{R}$ and $\bar{v}^1,\tilde{v}^{1}\in\mathcal{V}_{1}$, consider a small perturbation of \eqref{nas-def.q.e} around $v^{1}$ and denote by $P(\varepsilon)=  J'_1(h,v^1 + \varepsilon \bar{v}^1,v^2)\cdot(\tilde{v}^{1},0)$ such that
$$
  J'_1(h,v^1 + \varepsilon \bar{v}^1,v^2)\cdot(\tilde{v}^{1},0) = \alpha_1 \int_{\mathcal{O}_{1,d}\times (0,T)} (y^\varepsilon - y_{1,d})q^\varepsilon\ dxdt + \mu_1 \int_{\mathcal{O}_{1} \times (0,T)} (v^1 +\varepsilon\bar{v}^1)\tilde{v}^{1}\ dxdt, 
$$
where $y^{\varepsilon}$ is solution to
\begin{equation}\label{equationylambda}
		\left\{\begin{aligned}
			&y^{\varepsilon}_t - \left(a(x)\ell\left(\int_0^1 y^{\varepsilon} dx\right) y^{\varepsilon}_x\right)_x = h \mathds{1}_{\mathcal{O}} + (v^1+\varepsilon\bar{v}^{1}) \mathds{1}_{\mathcal{O}_1} + v^2 \mathds{1}_{\mathcal{O}_2} &&\text{in}&& Q,\\
			&y^\varepsilon(0,t)=y^\varepsilon(1,t)=0 &&\text{on} && (0,T), \\
			&y^\varepsilon(\cdot,0)=0 &&\text{in} && (0,1),
		\end{aligned}
		\right.
\end{equation}
and $q^\varepsilon$ the derivative of the state $y^{\varepsilon}$ with respect to $v^{1}$ in the direction $\tilde{v}^{1}$, i.e. the solution to 
\begin{equation}\label{eq:hat_y_lambda_PDE}
		\left\{\begin{aligned}
			&q_t^\varepsilon - \left( a(x) \left[ \ell\left(\int_0^1 y^\varepsilon dx\right) q_x^\varepsilon + \ell'\left(\int_0^1 y^\varepsilon dx\right) \int_0^1 q^\varepsilon dx \ y_x^\varepsilon \right] \right)_x=\tilde{v}^{1} \mathds{1}_{\mathcal{O}_1} &&\text{in}&& Q,\\
			&q^\varepsilon(0,t)=q^\varepsilon(1,t)=0 &&\text{on} && (0,T), \\
			&q^\varepsilon(\cdot,0)=0 &&\text{in} && (0,1),
		\end{aligned}
		\right.
\end{equation}
where we denote $y = y^\varepsilon|_{\varepsilon=0}$ and $q = q^\varepsilon|_{\varepsilon=0}$. 

The adjoint system of \eqref{eq:hat_y_lambda_PDE} is
\begin{equation}
  \label{eq:adjoint_of_hat_y_lambda}
		\left\{\begin{array}{lll}
			-\phi_t^\varepsilon - \left( a(x) \ell\left(\displaystyle\int_0^1 y^\varepsilon dx\right) \phi_x^\varepsilon \right)_{x}\\ + \ell'\left(\displaystyle\int_0^1 y^\varepsilon dx\right) \displaystyle\int_0^1 \ ( a(x) \ y_{x}^\varepsilon \ \phi_{x}^\varepsilon ) dx = \alpha_1 (y^\varepsilon-y_{1,d}) \mathds{1}_{\mathcal{O}_{1,d}} &\text{in}& Q,\\
			\phi^\varepsilon(0,t)=\phi^\varepsilon(1,t)=0 &\text{on} & (0,T), \\
			\phi^\varepsilon(\cdot,T)=0 &\text{in} & (0,1),
		\end{array}
		\right.
\end{equation}
where we denoting $\phi^{\varepsilon} = \phi^\varepsilon|_{\varepsilon=0}$.
Multiplying \eqref{eq:hat_y_lambda_PDE} by $\phi^\varepsilon$ and integrating over $Q$ we get
$$
\int_Q \left( q_t^\varepsilon - \left( a(x) \left[ \ell\left(\int_0^1 y^\varepsilon dx\right) q_x^\varepsilon + \ell'\left(\int_0^1 y^\varepsilon dx\right) \int_0^1 q^\varepsilon dx \ y_x^\varepsilon \right] \right)_x \right) \phi^\varepsilon = \int_Q \tilde{v}^{1} \mathds{1}_{\mathcal{O}_{1}} \phi^\varepsilon.
$$
Integrating by parts, and using \eqref{eq:adjoint_of_hat_y_lambda} we obtain 
$$
\alpha_1 \int_Q (y^\varepsilon-y_{1,d}) \mathds{1}_{\mathcal{O}_{1,d}} q^\varepsilon \ dxdt = \int_Q \tilde{v}^{1} \mathds{1}_{\mathcal{O}_1} \phi^\varepsilon \ dxdt, 
$$
and for $\varepsilon=0$,
$$
\alpha_1 \int_Q (y-y_{1,d}) \mathds{1}_{\mathcal{O}_{1,d}} q \ dxdt = \int_Q \tilde{v}^{1} \mathds{1}_{\mathcal{O}_{1}} \phi \ dxdt.
$$
Using the systems above we can consider
\begin{eqnarray}
\label{eq:DJ1_lambda-DJ1}
 P(\varepsilon)-P(0)&=&\alpha_1 \int_{\mathcal{O}_{1,d}\times (0,T)} (y^\varepsilon - y_{1,d}) q^\varepsilon \ dxdt-\alpha_1 \int_{\mathcal{O}_{1,d}\times (0,T)} (y - y_{1,d}) q \ dxdt\nonumber\\ &\quad&+\varepsilon \mu_1 \int_{\mathcal{O}_{1}\times (0,T)} \bar{v}^{1}\tilde{v}^{1} \ dxdt.
\end{eqnarray}
Dividing by $\varepsilon$ in \eqref{eq:DJ1_lambda-DJ1} becomes 
\begin{equation}\label{p1-p0lambda}
\frac{P(\varepsilon)-P(0)}{\varepsilon}= \int_{\mathcal{O}_{1} \times (0,T)} \left(\frac{\phi^\varepsilon - \phi}{\varepsilon}\right)\tilde{v}^{1} \ dxdt + \mu_1 \int_{\mathcal{O}_{1} \times (0,T)} \bar{v}^{1} \ \tilde{v}^{1} \ dxdt.
\end{equation}
Notice that, using the system (\ref{eq:adjoint_of_hat_y_lambda})
\begin{eqnarray}\label{differencebetweensystem1}
&\quad&-(\phi^{\varepsilon}-\phi)_{t} - \left( a(x) \ell\left(\int_0^1 y^\varepsilon dx\right) \phi^\varepsilon_{x} - a(x) \ell\left(\int_0^1 y \,dx\right) \phi_x \right)_x\\
&\quad&+\ell'\left(\int_0^1 y^\varepsilon dx\right) \int_0^1 ( a(x) y_{x}^\varepsilon  \phi_{x}^\varepsilon) dx- \ell'\left(\int_0^1 y \,dx\right) \int_0^1 ( a(x) y_{x}  \phi_{x}) dx  = \alpha_1 (y^\varepsilon - y) \mathds{1}_{\mathcal{O}_{1,d}},\nonumber
\end{eqnarray}
and using the system (\ref{equationylambda}) we obtain
\begin{equation}\label{differencebetweensystem2}
(y^{\varepsilon}-y)_{t}-\left( a(x) \ell\left(\int_0^1 y^\varepsilon dx\right) y_x^\varepsilon - a(x) \ell\left(\int_0^1 y\, dx\right) y_x \right)_x = \varepsilon\bar{v}^{1} \mathds{1}_{\mathcal{O}_1}.
\end{equation}
Consider the limits
\begin{equation}
\eta=\lim_{\varepsilon\to 0} \frac{1}{\varepsilon}(\phi^\varepsilon - \phi) \ \text{and} \ \theta=\lim_{\varepsilon\to 0} \frac{1}{\varepsilon}(y^\varepsilon - y).
\end{equation}
Dividing (\ref{differencebetweensystem1}) and (\ref{differencebetweensystem2}) by $\varepsilon$ and making $\varepsilon \to 0$ we can deduce the following system 
\begin{equation}
\label{eq:accoplatesystem}
		\left\{\begin{aligned}
			&-\eta_t - \ell\left(\int_0^1 y \,dx\right) \left( a(x) \eta_x \right)_x - \ell'\left(\int_0^1 y\, dx\right) \left( \int_0^1 \theta\, dx \right) \ (a(x) \phi_{x})_{x}\\& + \ell'\left(\int_0^1 y \,dx\right) \left[ \int_0^1 a(x) y_{x} \eta_{x} \,dx + \int_0^1 a(x) \theta_{x} \phi_{x}\, dx \right]\\ & + \ell''\left(\int_0^1 y \,dx\right) \left( \int_0^1 \theta \,dx \right) \int_0^1 a(x)y_{x} \phi_{x} \,dx =\alpha_1 \theta \mathds{1}_{\mathcal{O}_{1,d}} &&\text{in}&& Q,\\
            &\theta_t - \left( a(x)\ell\left(\int_0^1 y\, dx\right) \theta_x + a(x) \ell'\left(\int_0^1 y\, dx\right) \left( \int_0^1 \theta \,dx \right) y_x \right )_x= \bar{v}^{1} \mathds{1}_{\mathcal{O}_1}&&\text{in}&& Q,\\ 
			&\eta(0,t)=\eta(1,t)=0,\ \theta(0,t)=\theta(1,t)=0 &&\text{on} && (0,T), \\
			&\eta(\cdot,T)=0, \ \theta(\cdot,0)=0 &&\text{in} && (0,1).
		\end{aligned}
		\right.
\end{equation}
Then, when $\varepsilon \to 0$ in (\ref{p1-p0lambda}), we obtain the second derivative
\begin{equation}
\label{eq:D2J1_w1_w2}
 J''_1(h,v^1,v^2)\cdot (\bar{v}^{1}, \tilde{v}^{1})  = \int_{\mathcal{O}_{1}\times (0,T)} \eta \tilde{v}^{1} \, dxdt + \mu_1 \int_{\mathcal{O}_{1}\times (0,T)} \bar{v}^{1}\tilde{v}^{1} \, dxdt, \ \forall\bar{v}^{1},\tilde{v}^{1} \in \mathcal{V}_{1}. 
\end{equation}
In particular, taking $\bar{v}^{1}=\tilde{v}^{1}$
\begin{equation}
 J''_1(h,v^1,v^2)\cdot(\bar{v}^{1}, \bar{v}^{1})= \int_{\mathcal{O}_{1} \times (0,T)} \eta \bar{v}^{1} \, dxdt + \mu_1 \int_{\mathcal{O}_{1} \times (0,T)} |\bar{v}^{1}|^{2} \, dxdt. 
\end{equation}
Let us prove that there exists a constant $C> 0$ such that
\begin{equation}
\label{eq:integral_eta_w1_bounded}
\left|\int_{\mathcal{O}_1 \times (0,T)} \eta \bar{v}^{1} \ dxdt\right|\leq C(\|h\|_{L^{2}(0,T;L^{2}(\mathcal{O}))}+\|y_{0}\|_{{H^{1}_{a}(0,1)}}+1)\|\bar{v}^{1}\|^{2}_{\mathcal{V}_{1}}.
\end{equation}
We also get the following 
\begin{eqnarray*}   
\int_{\mathcal{O}_{1}\times (0,T)} \eta \bar{v}^{1} \, dxdt 
&=&\int_{Q} \eta\left[\theta_t - a(x)\left(\ell\left(\int_0^1 y \,dx\right)\theta_x+ \ell'\left(\int_0^1 y\, dx\right)\left(\int_0^1 \theta\, dx\right)y_{x}\right)_{x} \right] dxdt\\
&=&\int_{Q} \theta\left[- \eta_t - \left(a(x)\ell\left(\int_0^1 y\, dx\right) \eta_x \right)_x + \ell'\left(\int_0^1 y dx\right) \left(\int_0^1 a(x)\eta_{x}y_{x} dx\right) \right] dxdt\\
&=&\int_{Q}\alpha_{1}|\theta|^{2}\mathds{1}_{\mathcal{O}_{1,d}}\, dxdt+\int_{Q} \theta\left(\ell'\left(\int_{0}^{1}y\,dx\right)\left(\int_{0}^{1}\theta\,  dx\right)(a(x)\phi_{x})_{x}\right) dxdt\\
&\quad&-\int_{Q} \theta\left(\ell'\left(\int_{0}^{1}y\,dx\right)\int_{0}^{1}a(x)\theta_{x}\phi_{x}\,dx\right) dxdt\\
&\quad&-\int_{Q} \theta\left(\ell''\left(\int_{0}^{1}y\,dx\right)\left(\int_{0}^{1}\theta \,dx\right)\int_{0}^{1}a(x)y_{x}\phi_{x} \, dx\right)  dxdt.    
\end{eqnarray*}
From (\ref{eq:adjoint_of_hat_y_lambda}) with $\varepsilon = 0$ and (\ref{eq:accoplatesystem}), using the energy estimates of Appendix \ref{appendix C},  considering that $\phi$ verifies the same equation as $p^1$, we have

\begin{equation*}
\|(a\phi_{x})_x\|^{2}_{L^{2}(Q)}+\|\sqrt{a} \phi_{x}\|^{2}_{L^{\infty}(0,T,L^{2}(0,1))}\leq C\left( \|h\|^{2}_{L^{2}((0,T),L^2(\mathcal{O}))}+\|y_{0}\|_{H^{1}_{a}(0,1)}+\sum_{i=1}^{2}\|y_{i,d}\|^{2}_{L^{2}((0,T),L^2(\mathcal{O}_{i,d}))} \right),
\end{equation*}
\begin{equation*}
\|(ay_{x})_x\|^{2}_{L^{2}(Q)}+\|\sqrt{a} y_{x}\|^{2}_{L^{\infty}(0,T,L^{2}(0,1))}\leq C\left( \|h\|^{2}_{L^{2}((0,T),L^2(\mathcal{O}))}+\|y_{0}\|_{H^{1}_{a}(0,1)}+\sum_{i=1}^{2}\|y_{i,d}\|^{2}_{L^{2}((0,T),L^2(\mathcal{O}_{i,d}))} \right),
\end{equation*}
and analogous energy estimates for $\theta$, 
\begin{equation*}
\|\sqrt{a} \theta_{x}\|^{2}_{L^{2}(Q)}+\|\theta\|^{2}_{L^{\infty}(0,T,L^{2}(0,1))}\leq C\left(\|\bar{v}^{1}\|^{2}_{L^{2}(\mathcal{O}_{1,d}\times(0,T))}\right).
\end{equation*}
Thus,
\begin{eqnarray*}
\left|\int_{\mathcal{O}_1 \times (0,T)} \eta \tilde{v}^{1} \, dxdt\right|&\leq& C\left(\int_{0}^{T}\|\sqrt{a}\phi_{x}(t)\|\ \|\sqrt{a} y_{x}(t)\| \|\theta(t)\|^{2} dt \|\sqrt{a} \phi_{x}(t)\| dt \right. \\
&\quad& \left.+\int_{0}^{T} \|\sqrt{a} \theta_{x}(t)\|\|\theta(t)\|+\int_{0}^{T}\|\theta(t)\|^{2}dt\right)\\
&\leq&C\left(\|\sqrt{a} \phi_{x}\|_{L^{\infty}(0,T,L^{2}(0,1))}+\|\sqrt{a} \phi_{x}\|_{L^{\infty}(0,T,L^{2}(0,1))} \|\sqrt{a} y_{x}\|_{L^{\infty}(0,T,L^{2}(0,1))}\right)\\ 
&\quad&\times\int_0^T \left(\|\theta\|^2_{L^{2}(0,1)}+\|\sqrt{a} \theta_{x}\|^2_{L^{2}(0,1)}\right)dt\\
&\leq&C\left(\|h\|_{L^{2}(0,T;L^{2}(\mathcal{O}))}+\|y_{0}\|_{H^{1}_{a}(0,1)}^{2} +\sum_{i=1}^{2}\|y_{i,d}\|^{2}_{L^{2}((0,T),L^2(\mathcal{O}_{i,d}))} +1 \right)\| \bar{v}^{1}\|^{2}_{\mathcal{V}_{1}}.
\end{eqnarray*}
To finish the proof, we have that \eqref{eq:integral_eta_w1_bounded} implies
$$
J''_1(h,v^1,v^2)\cdot(\bar{v}^{1}, \bar{v}^{1})\geq (\mu_1 - C) \int_{\mathcal{O}_1 \times (0,T)} |\bar{v}^{1}|^2 \ dxdt.
$$
Taking $\mu_1$ big enough, this shows that $J_1$ is convex. Analogously, by a similar computation, taking $\mu_2$ big enough, we have that $J_2$ is convex. Thus the Nash quasi-equilibrium par $(v^1,v^2)$ is in fact a Nash equilibrium.
\end{proof}

\section{Null controllability of the linearized system}
\label{Sec:null_linearized_system}
\subsection{Preliminaries}
\label{subsec:prelim}
As a first step we study controllability of the linearized system about the zero solution.

Recall the optimality system \eqref{eq:optimality_system}. We compute formally the derivative a zero of the map
$T(y,p^1,p^2) = (T_0(y,p^1,p^2),T_1(y,p^1,p^2),T_2(y,p^1,p^2))$, where
\begin{equation*}
T_0(y,p^1,p^2) = y_t - \left(a(x)\ell\left(\int_0^1 y \, dx\right) y_x\right)_x +\ \frac{1}{\mu_1} p^1 \mathds{1}_{\mathcal{O}_{1}} + \frac{1}{\mu_2} p^2 \mathds{1}_{\mathcal{O}_{2}},    
\end{equation*}
and, for $i=\lbrace1,2\rbrace$,
\begin{equation*}
T_i(y,p^1,p^2) = - p_t^i - \left( a(x) \ell\left(\int_0^1 y \, dx\right) p_x^i \right)_x  + \ \ell'\left(\int_0^1 y\, dx\right) \int_0^1 ( a(x) y_x \ p^i_x)\, dx - \alpha_i y \mathds{1}_{\mathcal{O}_{i,d}}.    
\end{equation*}
Thus, 
\begin{equation*}
DT(0,0,0)\cdot(\tilde y,\tilde p^1, \tilde p^2) = \lim_{\lambda \to 0} \frac{T(\lambda(\tilde y,\tilde p^1,\tilde p^2))-T(0,0,0)}{\lambda} = \lim_{\lambda \to 0} \frac{T(\lambda(\tilde y,\tilde p^1,\tilde p^2))}{\lambda}    
\end{equation*}
Forgetting the tilde in the notations, 
we get the linearized system, for $i=\lbrace 1,2\rbrace$,
\begin{equation}\label{eq:linearized_system}
\left\{\begin{aligned}
&y_t - (\ell(0) a(x) y_x)_x = h \mathds{1}_{\mathcal{O}} -\frac{1}{\mu_1} p^1 \mathds{1}_{\mathcal{O}_{1}} - \frac{1}{\mu_2}p^2 \mathds{1}_{\mathcal{O}_{2}} + H &&\text{in}&& Q, \\
&-p_t^i - \left(\ell(0) a(x) p_x^i \right)_x  = \alpha_i y \mathds{1}_{O_{i,d}} + H_i  &&\text{in}&& Q, \\
&y(0,t)=y(1,t)=0,\ \ p^i(0,t) = p^i(1,t) = 0 &&\text{on}&& (0,T), \\
&y(\cdot,0) = y_{0},\ \ p^i(\cdot,T) = 0  &&\text{in}&& (0,1),
\end{aligned}
\right.
\end{equation}
with associated adjoint system
\begin{equation}
\label{eq:adjoint_optimality_system}
\left\{\begin{aligned}
&-\phi_t - (\ell(0) a(x) \phi_x)_x = F + \alpha_1 \psi^1 \mathds{1}_{\mathcal{O}_{1,d}} + \alpha_2 \psi^2 \mathds{1}_{\mathcal{O}_{2,d}} && \text{in }&& Q, \\
&\psi^i_t - (\ell(0) a(x) \psi^i_x)_x = F_i -\frac{1}{\mu_i} \phi \mathds{1}_{\mathcal{O}_{i}}  && \text{in} && Q, \\
&\phi(0,t)=\phi(1,t)=0, \ \ \psi^i(0,t)=\psi^i(1,t)=0  && \text{on} && (0,T), \\
&\phi(\cdot,T) = \phi^T, \ \ \psi^i(\cdot,0) = 0 && \text{in} && (0,1), \\
\end{aligned}
\right.
\end{equation}
where $\phi^T \in L^2(0,1)$. To simplify the system, we assume the condition
\begin{equation}
\label{eq:condition_O1d_O2d}
\mathcal{O}_{1,d} = \mathcal{O}_{2,d} = \mathcal{O}_{d}.
\end{equation}
Introducing the new variable $\varrho = \alpha_1 \psi^1 + \alpha_2 \psi^2$, the adjoint system becomes 
\begin{equation}
\label{eq:adjoint_optimality_system2}
\left\{\begin{aligned}
&-\phi_t - (\ell(0) a(x) \phi_x)_x = \varrho \mathds{1}_{\mathcal{O}_d} + G_0 && \text{in} && Q, \\
&\varrho_t - (\ell(0) a(x) \varrho_x)_x = - \left( \frac{\alpha_1}{\mu_1}\mathds{1}_{\mathcal{O}_1} + \frac{\alpha_2}{\mu_2} \mathds{1}_{\mathcal{O}_2} \right)\phi + \tilde{G}  && \text{in } && Q, \\
&\phi(0,t)=\phi(1,t)=0, \ \ \varrho(0,t)=\varrho(1,t)=0 && \text{on} && (0,T), \\
&\phi(\cdot,T)=\phi^T \ \ \varrho(\cdot,0) = 0 && \text{in} && (0,1), 
\end{aligned}
\right.
\end{equation}
where $\phi^T \in L^2(0,1)$ and $\tilde{G}=\alpha_{1}F_{1}+\alpha_{2}F_{2}$.

The following result presents a  well-posedness for the linear system \eqref{eq:linearized_system}. This will be used later in the definition of our admissible state space $\mathcal{Y}$, defined in Section \ref{control for nonlinear system}.

\begin{proposition}\label{Regularidade para linear system}
For all $H, H_{1}, H_{2}\in L^{2}(Q)$ and $y_{0}\in L^{2}(0,1)$, there exists a unique weak solution $y, p^{1}, p^{2}\in C^{0}([0,T];L^{2}(0,1))\cap L^{2}(0,T;H^{1}_{a}(0,1))$ of \eqref{eq:linearized_system} and there exists a positive constant $C=C(T)$ such that 
\begin{equation*}
    \begin{array}{l}
\displaystyle\sup_{t\in[0,T]}\left(\|y(t)\|^{2}_{L^{2}(0,1)} + \|p^{1}(t)\|^{2}_{L^{2}(0,1)} + \|p^{2}(t)\|^{2}_{L^{2}(0,1)}\right) \vspace{0.1cm}\\
        + \displaystyle\int_{0}^{T}\left(\|\sqrt{a}y_{x}\|^{2}_{L^{2}(0,1)} + \|\sqrt{a}p^{1}_{x}\|^{2}_{L^{2}(0,1)} + \|\sqrt{a}p^{2}_{x}\|^{2}_{L^{2}(0,1)} \right) dt\\
        \leq C\left(\|y_{0}\|^{2}_{L^{2}(0,1)} + \|H\|_{L^{2}(Q)} + \|H_{1}\|_{L^{2}(Q)} + \|H_{2}\|_{L^{2}(Q)}  \right).
    \end{array}
\end{equation*}
Moreover, if $y_{0}\in H^{1}_{a}(0,1)$, then 
\begin{equation*}
    y, p^{1}, p^{2}\in H^{1}(0,T; L^{2}(0,1))\cap L^{2}(0,T;H^{2}_{a}(0,1))\cap C^{0}([0,T];H^{1}_{a}(0,1)),
\end{equation*}
and there exists a positive constant $C=C(T)$ such that 
\begin{equation*}
    \begin{array}{l}
\displaystyle\sup_{[0,T]}\left( \|y(t)\|^{2}_{H^{1}_{a}(0,1)} + \|p^{1}(t)\|^{2}_{H^{1}_{a}(0,1)} + \|p^{2}(t)\|^{2}_{H^{1}_{a}(0,1)}\right)\\
+ \displaystyle\int_{0}^{T}\left( \|y_{t}\|^{2}_{L^{2}(0,1)} + \|p^{1}_{t}\|^{2}_{L^{2}(0,1)}  + \|p^{2}_{t}\|^{2}_{L^{2}(0,1)} \right)dt\\
+ \displaystyle\int_{0}^{T}\left(\|(ay_{x})_{x}\|^{2}_{L^{2}(0,1)} + \|(ap^{1}_{x})_{x}\|^{2}_{L^{2}(0,1)}+\|(ap^{2}_{x})_{x}\|^{2}_{L^{2}(0,1)}  \right)dt\\
\leq C\left( \|y_{0}\|^{2}_{H^{1}_{a}(0,1)} + \|H\|_{L^{2}(Q)} + \|H_{1}\|_{L^{2}(Q)}+\|H_{2}\|_{L^{2}(Q)}  \right).
    \end{array}
\end{equation*}
\end{proposition}
\begin{proof}
    See Proposition 2.1 of \cite{FadiliManiar}. 
\end{proof}
We observe that, by taking $\ell(0) = 1$, \eqref{eq:adjoint_optimality_system2} coincides with the linearized system studied in \cite{djomegne2022hierarchical} (where a semilinear term is also included), together with the corresponding Carleman inequalities and observability result. As a consequence, Proposition 6.1 in \cite{djomegne2022hierarchical} ensures that the linearized system \eqref{eq:linearized_system} is null controllable. However, in order to apply the Liusternik inverse mapping theorem to obtain the controllability of the nonlinear problem, we follow \cite{DemarqueLimacoViana_deg_sys2020} for the choice of Carleman weights and for additional estimates required for the controllability of the linearized system.

\subsection{Carleman estimates}
Following \cite{DemarqueLimacoViana_deg_eq2018, DemarqueLimacoViana_deg_sys2020}, let $\mathcal{O}'=(\alpha',\beta') \subset\subset \mathcal{O}$ and let $\Psi : [0,1] \to \R$ be a $C^2$ function such that
\begin{equation*}
\Psi(x) = 
\begin{cases}
\displaystyle\int_0^x \frac{s}{a(s)} ds, \quad x \in [0,\alpha'), \\
-\displaystyle\int_{\beta'}^x \frac{s}{a(s)} ds, \quad x \in [\beta',1].
\end{cases}    
\end{equation*}
For $\lambda \geq \lambda_0$ define the functions
\begin{equation*}
\left\{\begin{aligned}
&\Theta(t) = \frac{1}{(t(T-t))^4}, \qquad \vartheta(x) = e^{\lambda(|\Psi|_\infty + \Psi)}, \qquad \xi(x,t) = \Theta(t) \vartheta(x) \\ &\varphi(x,t) = \Theta(t) (e^{\lambda(|\Psi|_\infty + \Psi)}-e^{3\lambda|\Psi|_\infty}).    
\end{aligned}
\right.
\end{equation*}

We introduce the notation
\begin{equation*}
I(\varsigma) = \int_Q e^{2s\varphi}((s\lambda)\xi a \varsigma_x^2 + (s\lambda)^2 \xi^2 \varsigma^2)dxdt.
\end{equation*}

The following proposition was proved in \cite{DemarqueLimacoViana_deg_sys2020} for an adjoint system of two retrograde equations. After making the substitution $t \mapsto T-t$ to our second equation in \eqref{eq:adjoint_optimality_system2}, we are in the case treated in \cite{DemarqueLimacoViana_deg_sys2020} and the conclusion remains true after coming back to the original variable.

\begin{proposition} 
There exist positive constants $C$, $\lambda_0$ and $s_0$ such that, for any $s \geq s_0$, $\lambda \geq \lambda_0$ and any $\phi^T \in L^2(Q)$, the corresponding solution $(\phi,\varrho)$ of \eqref{eq:adjoint_optimality_system2} satisfies
$$
I(\phi) + I(\varrho) \leq C \left( \int_Q e^{2s\varphi} s^4 \lambda^4 \xi^4 (|G_0|^2+|\tilde{G}|^2)dxdt + \int_{\mathcal{O}\times(0,T)} e^{2s\varphi} s^8 \lambda^8 \xi^8 \phi^2dxdt \right).
$$
\end{proposition}
\begin{proof}
    See Proposition 3 of \cite{DemarqueLimacoViana_deg_sys2020}.
\end{proof}

In order to get the global null controllability of the linearized system, we need a Carleman inequality with weights which do not vanish at $t=0$. Consider the function $m \in C^\infty([0,T])$ satisfying $m(0)>0$,
\begin{equation}
\label{eq:def_m}
m(t) \geq t^4(T-t)^4, \quad t \in (0,T/2], \qquad\qquad m(t) = t^4(T-t)^4, \quad t \in [T/2,T],    
\end{equation}
and define
$$
\tau(t) = \frac{1}{m(t)}, \qquad \zeta(x,t) = \tau(t) \vartheta(x), \qquad A(x,t)  = \tau(t) (e^{\lambda(|\Psi|_\infty + \Psi)}-e^{3\lambda|\Psi|_\infty}).
$$
Define
$$
\Gamma_0(\phi,\varrho) = \int_Q e^{2s A}((s\lambda) \zeta a (|\phi_x|^2 + |\varrho_x|^2) + (s\lambda)^2 {\zeta}^2 (|\phi|^2 + |\varrho|^2))dxdt,
$$
$$
\Gamma(\phi,\varrho,\tilde \varrho) = \int_Q e^{2s A}((s\lambda) \zeta a (|\phi_x|^2+|\varrho_x|^2+|\tilde \varrho_x|^2) + (s\lambda)^2 {\zeta}^2 (|\phi|^2 + |\varrho|^2 + |\tilde \varrho|^2))dxdt.
$$

\begin{proposition} 
There exist positive constants $C$, $\lambda_0$ and $s_0$ such that, for any $s \geq s_0$, $\lambda \geq \lambda_0$ and any $\phi^T\in L^2(Q)$, the corresponding solution $(\phi,\varrho)$ of \eqref{eq:adjoint_optimality_system2} satisfies
$$
\Gamma_0(\phi,\varrho) \leq C \left( \int_Q e^{2s A} s^4 \lambda^4 \zeta^4 (|G_0|^2+|\tilde{G}|^2)dxdt + \int_{\mathcal{O}\times(0,T)} e^{2s A} s^8 \lambda^8 \zeta^8 |\phi|^2 dxdt\right).
$$
\end{proposition}
\begin{proof}
    See Proposition 4 of \cite{DemarqueLimacoViana_deg_sys2020}.
\end{proof}

Since $\varrho = \alpha_1 \psi^1 + \alpha_2 \psi^2$ and $\tilde{G}=\alpha_{1}F_{1}+\alpha_{2}F_{2}$ we get directly the following Carleman inequality for the original adjoint system \eqref{eq:adjoint_optimality_system}.

\begin{proposition}
\label{prop:carleman}
There exist positive constants $C$, $\lambda_0$ and $s_0$ such that, for any $s \geq s_0$, $\lambda \geq \lambda_0$ and any $\phi^T \in L^2(Q)$, the corresponding solution $(\phi,\psi^1,\psi^2)$ of \eqref{eq:adjoint_optimality_system} satisfies
\begin{equation}
\label{Carleman for eq:adjoint_optimality_system}
\Gamma(\phi,\psi^1,\psi^2) \leq C \left( \int_Q e^{2s A} s^4 \lambda^4 \zeta^4 (|F|^2+|F_1|^2+|F_2|^2)dxdt + \int_{\mathcal{O}\times(0,T)} e^{2s A} s^8 \lambda^8 \zeta^8 |\phi|^2 dxdt\right).
\end{equation}
\end{proposition}

As a corollary we get the observability inequality.
\begin{corollary} 
\label{cor:observability}
There exist positive constants $C$, $\lambda_0$ and $s_0$ such that, for any $s \geq s_0$, $\lambda \geq \lambda_0$ and any $\phi^T\in L^2(Q)$, the corresponding solution $(\phi,\varrho)$ of \eqref{eq:adjoint_optimality_system2} with $F_0 = F_1 = 0$, satisfies
\begin{equation}\label{Observability}
\|\phi(0)\|^2_{L^2(0,1)} + \|\varrho(T)\|^2_{L^2(0,1)} \leq C \int_{\mathcal{O}\times(0,T)} e^{2s A} s^8 \lambda^8 \zeta^8 |\phi|^2dxdt.
\end{equation}
\end{corollary}
\begin{proof}
    See Corollary 1 of \cite{DemarqueLimacoViana_deg_sys2020}.
\end{proof}
\color{black}

In the following sections we will need weights which depend only on $t$. Let us define
\begin{equation*}
\begin{cases}
A^*(t) = \displaystyle\max_{x \in (0,1)} A(x,t), \qquad \hat A(t) = \displaystyle\min_{x \in (0,1)} A(x,t), \\
\zeta^*(t) = \displaystyle\max_{x \in (0,1)} \zeta(x,t), \qquad \hat \zeta(t) = \min_{x \in (0,1)} \zeta(x,t),
\end{cases}
\end{equation*}
and we observe that $A^*(t) < 0$, $\hat A(t) < 0$ and that $\zeta^*(t) / \hat \zeta(t)$ does not depend on $t$ and is equal to some constant $\zeta_0 \in \R$. Moreover, if $\lambda$ is sufficiently large we can suppose
\begin{equation}
\label{eq:comp_pesos}
3 A^*(t) < 2 \hat A(t) < 0.
\end{equation}
Let us define 
$$
\hat \Gamma(\phi,\varrho,\tilde \varrho) = \int_Q e^{2s \hat A}[(s\lambda) \hat \zeta a (|\phi_x|^2+|\varrho_x|^2+|\tilde \varrho_x|^2) + (s\lambda)^2 {\hat \zeta}^2 (|\phi|^2 + |\varrho|^2 + |\tilde \varrho|^2)]dxdt.
$$
Thus, Proposition \ref{prop:carleman} and Corollary \ref{cor:observability}  imply directly the following corollary where the weights depend only on $t$.

\begin{corollary} 
\label{cor:carleman_pesos_t}
There exist positive constants $C$, $\lambda_0$ and $s_0$ such that, for any $s \geq s_0$, $\lambda \geq \lambda_0$ and any $\phi^T \in L^2(Q)$, the corresponding solution $(\phi,\psi^1,\psi^2)$ of \eqref{eq:adjoint_optimality_system} satisfies
\begin{equation}\label{pesos_t_Carleman for eq:adjoint_optimality_system}
\begin{split}
\|\phi(0)\|^2_{L^2(0,1)} + 
\hat \Gamma(\phi,\psi^1,\psi^2) \leq & C \left( \int_Q e^{2s A^*} (\zeta^*)^4 (|F|^2+|F_1|^2+|F_2|^2)dxdt \right. \\
& \left. + \int_{\mathcal{O}\times(0,T)} e^{2s A^*} (\zeta^*)^8 |\phi|^2 dxdt\right).
\end{split}
\end{equation}
\end{corollary}

\subsection{Global null controllability of the linearized system}

Lets us define the weight funtions: 
\begin{equation}\label{eq:weights_rhos}
   \left\{ \begin{array}{l}
 \rho_0 = e^{-sA^*}  (\zeta^*)^{-2}, \qquad   \rho_1 = e^{-sA^*}  (\zeta^*)^{-4},\\  \rho_2 = e^{-3sA^*/2}  \hat \zeta^{-1},  \qquad \hat{\rho} = e^{-sA^*} (\zeta^*)^{-3},      
    \end{array}\right.
\end{equation}
satisfying
\begin{equation}\label{eq:compara_rhos}
 \rho_{1}\leq C \hat{\rho}\leq C\rho_{0}\leq C\rho_{2} \qquad \text{and} \qquad \hat{\rho}^{2}=\rho_{1}\rho_{0}.   
\end{equation}
In particular, Corollary \ref{cor:carleman_pesos_t} and \eqref{eq:comp_pesos} imply 
\begin{equation}
\label{eq:carleman_simples}
\begin{split}
\|\phi(0)\|^2_{L^2(0,1)} +
\int_Q \rho_2^{-2} (|\phi|^2 + |\psi^{1}|^2 + |\psi^{2}|^2)dxdt
\leq & \ C \left( \int_Q \rho_0^{-2} (|F|^2+|F_1|^2+|F_2|^2)dxdt \right. \\
& \left. + \int_{\mathcal{O}\times(0,T)} \rho_1^{-2} |\phi|^2 dxdt\right).
\end{split}
\end{equation}

In the following theorem we show the global null controllability of the linearized system \eqref{eq:linearized_system}. In particular, since the weight $\rho_0$ blows up at $t=T$, \eqref{estimate for solution} shows that $y(\cdot,T)=0$ in $[0,1]$. The same vanishing conclusion holds for $p^i$, $i=1,2$ (which define the follower's controls $v^i$) and the leader control $h$.
Furthermore, estimates \eqref{des Proposition 5} and \eqref{des Proposition 6} give additional estimates verified by the derivatives of the states $(y,p^1,p^2)$ that are needed for the local null controllability of the nonlinear system
in Section \ref{control for nonlinear system}.

\begin{theorem}\label{theorem case linear}
If $y_{0}\in L^{2}(0,1)$, $\rho_2 H$, $\rho_2 H_1$, $\rho_2 H_2 \in L^2(Q)$, then there exists a control $h\in L^{2}(\mathcal{O}\times (0,T))$ with associated states $y, p^{1}, p^{2}\in C^{0}([0,T];L^{2}(0,1))\cap L^{2}(0,T;H^{1}_{a}(0,1))$  from \eqref{eq:linearized_system} such that
\begin{equation}\label{estimate for solution}
\int_Q \rho_0^2 |y|^2 dxdt+ \int_Q \rho_0^2 (|p^1|^2 +  |p^2|^2)dxdt + \int_{\mathcal{O} \times (0,T)} \rho_1^2 |h|^2dxdt \leq
C \kappa_{0}(H,H_{1},H_{2},y_{0}),   
\end{equation}
where $\kappa_{0}(H,H_{1},H_{2},y_{0})= \|\rho_2 H\|^2_{L^2(Q)} + \|\rho_2 H_1\|^2_{L^2(Q)} + \|\rho_2 H_2\|^2_{L^2(Q)} + \|y_{0}\|^2_{L_2(0,1)}$. In particular, $y(x,T)=0$, for all $x\in [0,1]$. 

Furthermore, 
\begin{equation}\label{des Proposition 5}
    \begin{array}{c}
\displaystyle\sup_{[0,T]}(\hat{\rho}^{2}\|y\|^{2}_{L^{2}(0,1)}) + \displaystyle\sup_{[0,T]}(\hat{\rho}^{2}\|p^{1}\|^{2}_{L^{2}(0,1)}) + \displaystyle\sup_{[0,T]}(\hat{\rho}^{2}\|p^{2}\|^{2}_{L^{2}(0,1)})\vspace{0.1cm}\\
+\displaystyle\int_{Q}\hat{\rho}^{2}\ell(0)a(x)(|y_{x}|^{2} + |p^{1}_{x}|^{2}+|p^{2}_{x}|^{2})dxdt\ \leq C \kappa_{0}(H,H_{1},H_{2},y_{0})  
    \end{array}
\end{equation}
and, if $y_{0}\in H^{1}_{a}(0,1)$ 
\begin{equation}\label{des Proposition 6}
    \begin{array}{c}
\displaystyle\sup_{[0,T]}(\rho_{1}^{2}\|\sqrt{a}y_{x} \|^{2}_{L^{2}(0,1)})+ \displaystyle\sup_{[0,T]}(\rho_{1}^{2}\|\sqrt{a}p^{1}_{x} \|^{2}_{L^{2}(0,1)}) + \displaystyle\sup_{[0,T]}(\rho_{1}^{2}\|\sqrt{a}p^{2}_{x} \|^{2}_{L^{2}(0,1)})\\
+ \displaystyle\int_{Q}\rho_{1}^{2}(|y_{t}|^{2}+|p^{1}_{t}|^{2}+|p^{2}_{t}|^{2}+|(\ell(0)a(x)y_{x})_{x}|^{2} + |(\ell(0)a(x)p^{1}_{x})_{x}|^{2}+ |(\ell(0)a(x)p^{2}_{x})_{x}|^{2})dxdt\\
\leq C \kappa_{1}(H,H_{1},H_{2},y_{0}),  
    \end{array}
\end{equation}
where $\kappa_{1}(H,H_{1},H_{1},y_{0})= \|\rho_2 H\|^2_{L^2(Q)} + \|\rho_2 H_1\|^2_{L^2(Q)} + \|\rho_2 H_2\|^2_{L^2(Q)} + \|y_{0}\|^2_{H^{1}_{a}(0,1)}$. 
\end{theorem}
\begin{proof}
Let us denote by $L\varphi =  \varphi_{t} - (\ell(0)a(x)\varphi_{x})_{x}$ and $L^{\ast}\varphi=-\varphi_{t} - (\ell(0)a(x)\varphi_{x})_{x}$. Then, we define 
\begin{equation*}
\begin{array}{l}
    \mathcal{P}_{0} = \{ (\phi,\psi^{1},\psi^{2})\in C^{2}(\overline{Q})^{3}; \phi(0,t)=\phi(1,t) = 0\ \text{a.e in}\ (0,T),\ \psi^{i}(0,t)=\psi^{i}(1,t)=0\, \text{a.e in}\, (0,T),\\ \hspace{1.2cm}\psi^{i}(\cdot,0)=0\ \text{in}\ (0,1), \ i=\lbrace 1, 2\rbrace\}
\end{array}
\end{equation*}
and the application $b:\mathcal{P}_{0}\times \mathcal{P}_{0}\rightarrow \mathbb{R}$ given by
\begin{eqnarray}\label{eq:def_b}
b((\tilde{\phi},\tilde{\psi^{1}},\tilde{\psi^{2}}),(\phi,\psi^{1},\psi^{2}))&=&\int_{Q}\rho_{0}^{-2}(L^{\ast}\tilde{\phi}-\alpha_{1}\tilde{\psi^{1}}\mathds{1}_{\mathcal{O}_{d}}-\alpha_{2}\tilde{\psi^{2}}\mathds{1}_{\mathcal{O}_{d}})(L^{\ast}\phi-\alpha_{1}\psi^{1}\mathds{1}_{\mathcal{O}_{d}}-\alpha_{2}\psi^{2}\mathds{1}_{\mathcal{O}_{d}}) dx dt\nonumber\\
&\quad&+\sum_{i=1}^{2}\displaystyle\int_{Q}\rho_{0}^{-2}(L\tilde{\psi^{i}} + \frac{1}{\mu_{i}}\tilde{\phi}\mathds{1}_{\mathcal{O}_{i}})(L\psi^{i} + \frac{1}{\mu_{i}}\phi\mathds{1}_{\mathcal{O}_{i}}) dx dt\nonumber\\
&\quad&+\int_{\mathcal{O}\times (0,T)}\rho_{1}^{-2}\tilde{\phi}\phi\ dx dt,\,\, \forall (\phi,\psi^{1},\psi^{2}),(\tilde{\phi},\tilde{\psi^{1}},\tilde{\psi^{2}})\in \mathcal{P}_{0},
\end{eqnarray}
which is bilinear on $\mathcal{P}_{0}$ and defines an inner product. Indeed, taking $(\tilde{\phi},\tilde{\psi^{1}},\tilde{\psi^{2}})=(\phi,\psi^{1},\psi^{2})$ in \eqref{eq:def_b}, we have that, by \eqref{eq:carleman_simples}  $b(\cdot,\cdot)$ is positive definite. The other properties are straightforwardly verified.

Let us consider the space $\mathcal{P}$ the completion of $\mathcal{P}_{0}$ for the norm associated to $b(\cdot,\cdot)$ (which we denote by $\|.\|_{\mathcal{P}}$). Then, $b(\cdot,\cdot)$ is symmetric, continuous and coercive bilinear form on $\mathcal{P}$.

Now, let us define the functional linear $\mathcal{S} :\mathcal{P}\rightarrow\mathbb{R}$ as 
\begin{equation*}
    \langle\mathcal{S}, (\phi,\psi^{1},\psi^{2})\rangle = \displaystyle\int_{0}^{1}y_{0}\phi(0)dx + \displaystyle\int_{Q}(H\phi + H_{1}\psi^{1} + H_{2}\psi^{2})dx dt.
\end{equation*}
Note that $\mathcal{S}$ is a bounded linear form on $\mathcal{P}$. Indeed, applying the classical Cauchy-Schwartz inequality 
in $\mathbb{R}^{4}$ and using \eqref{eq:carleman_simples}, we get 
\begin{equation}\label{l limitado}
    \begin{array}{l}
       |\langle\mathcal{S}, (\phi,\psi^{1},\psi^{2})\rangle|\leq  \|y_{0}\|_{L^{2}(0,1)}\|\phi(0)\|_{L^{2}(0,1)} + \|\rho_{2}H\|_{L^{2}(Q)}\|\rho_{2}^{-1}\phi\|_{L^{2}(Q)} \\
       \hspace{3cm} + \|\rho_{2}H_{1}\|_{L^{2}(Q)}\|\rho_{2}^{-1}\psi^{1}\|_{L^{2}(Q)} + \|\rho_{2}H_{2}\|_{L^{2}(Q)}\|\rho_{2}^{-1}\psi^{2}\|_{L^{2}(Q)}\\
       \leq C\left(\|y_{0}\|^{2}_{L^{2}(0,1)} + \|\rho_{2}H\|^{2}_{L^{2}(Q)} + \|\rho_{2}H_{1}\|^{2}_{L^{2}(Q)}  + \|\rho_{2}H_{2}\|^{2}_{L^{2}(Q)}\right)^{1/2}\left(b((\phi,\psi^{1},\psi^{2}),(\phi,\psi^{1},\psi^{2}))\right)^{1/2}\\
     \leq C\left(\|y_{0}\|^{2}_{L^{2}(0,1)} + \|\rho_{2}H\|^{2}_{L^{2}(Q)} + \|\rho_{2}H_{1}\|^{2}_{L^{2}(Q)}  + \|\rho_{2}H_{2}\|^{2}_{L^{2}(Q)}\right)^{1/2}\|(\phi,\psi^{1},\psi^{2})\|_{\mathcal{P}},
    \end{array}
\end{equation}
for all $(\phi,\psi^{1},\psi^{2})\in\mathcal{P}$. Consequently, in view of Lax-Milgram's lemma, there is only one $(\hat{\phi},\hat{\psi^{1}},\hat{\psi^{2}})\in\mathcal{P}$ satisfying
\begin{equation}\label{eq: por Lax-M.}    b((\hat{\phi},\hat{\psi^{1}},\hat{\psi^{2}}),(\phi,\psi^{1},\psi^{2})) = \langle\mathcal{S}, (\phi,\psi^{1},\psi^{2})\rangle,\,\, \forall (\phi,\psi^{1},\psi^{2})\in\mathcal{P}.
\end{equation}
Let us set
\begin{equation}\label{definição de y, pi, h}
    \left\{\begin{aligned}
      &y = \rho_{0}^{-2}(L^{\ast}\hat{\phi}-\alpha_{1}\hat{\psi^{1}}\mathds{1}_{\mathcal{O}_{d}}-\alpha_{2}\hat{\psi^{2}}\mathds{1}_{\mathcal{O}_{d}})&&\text{in}&& Q,\\
      &p^{i} = \rho_{0}^{-2}(L\hat{\psi^{i}} + \dfrac{1}{\mu_{i}}\hat{\phi}\mathds{1}_{\mathcal{O}_{i}}),\, i=\{1,2\} &&\text{in}&& Q,\\
      &h = -\rho_{1}^{-2}\hat{\phi}\mathds{1}_{\mathcal{O}} &&\text{in}&& Q.
\end{aligned}\right.
\end{equation}
Then, replacing \eqref{definição de y, pi, h} in \eqref{eq: por Lax-M.} we have
\begin{eqnarray*}
         \int_{Q}y B\,dxdt +\int_{Q}(p^{1}B_{1} + p^{2}B_{2})\,dxdt
         &=& \int_{0}^{1}y_{0}\phi(0)dx + \int_{\mathcal{O}\times(0,T)} h \phi\,dxdt\\ &\quad&+\displaystyle\int_{Q}(H\phi + H_{1}\psi^{1} + H_{2}\psi^{2})dxdt,
\end{eqnarray*}
where $(\phi,\psi^{1},\psi^{2})$ is a solution of the system
\begin{equation*}
\left\{\begin{aligned}
       &L^{\ast}\phi = B +\alpha_{1}{\psi^{1}}\mathds{1}_{\mathcal{O}_{d}} + \alpha_{2}{\psi^{2}}\mathds{1}_{\mathcal{O}_{d}}  &&\text{in}&& Q,  \\
       &L\psi^{i} = B_{i} - \dfrac{1}{\mu_{i}}{\phi}\mathds{1}_{\mathcal{O}_{i}}  \  &&\text{in}&& Q, \\
       &\phi(0,t)=\phi(1,t)=0, \ \ \psi^i(0,t)=\psi^i(1,t)=0 &&\text{on}&& (0,T),\\
       &\phi(\cdot,T)=0,\, \psi^{i}(\cdot,0)=0 &&\text{in}&& (0,1),
   \end{aligned}\right.
\end{equation*}
for $i=\lbrace 1, 2\rbrace$. Therefore, $(y,p^{1},p^{2})$ is a solution by transposition of \eqref{eq:linearized_system}. Also, as $(\hat{\phi},\hat{\psi^{1}},\hat{\psi^{2}})\in\mathcal{P}$ and 
$H, H_{1}, H_{2}\in L^{2}(Q)$,  
using the Proposition \ref{Regularidade para linear system}, we obtain
$$y,p^{1},p^{2}\in C^{0}([0,T];L^{2}(0,1))\cap L^{2}(0,T;H^{1}_{a}(0,1)).$$
Moreover, from \eqref{l limitado} and \eqref{eq: por Lax-M.}
\begin{equation*}
\left(b((\hat{\phi},\hat{\psi^{1}},\hat{\psi^{2}}),(\hat{\phi},\hat{\psi^{1}},\hat{\psi^{2}}))\right)^{1/2}  \leq C\left(\|y_{0}\|^{2}_{L^{2}(0,1)} + \|\rho_{2}H\|^{2}_{L^{2}(Q)} + \|\rho_{2}H_{1}\|^{2}_{L^{2}(Q)}  + \|\rho_{2}H_{2}\|^{2}_{L^{2}(Q)}\right)^{1/2}
\end{equation*}
that is,
\begin{equation*}
\begin{array}{l}
\displaystyle\int_{Q}\rho_{0}^{2}|y|^{2}\ dxdt       + \displaystyle\sum_{i=1}^{2}\displaystyle\int_{Q}\rho_{0}^{2}|p^{i}|^{2}\ dxdt +\displaystyle\int_{\mathcal{O}\times (0,T)}\rho_{1}^{2}|h|^{2} dx dt\\
\leq C\left(\|y_{0}\|^{2}_{L^{2}(0,1)} + \|\rho_{2}H\|^{2}_{L^{2}(Q)} + \|\rho_{2}H_{1}\|^{2}_{L^{2}(Q)}  + \|\rho_{2}H_{2}\|^{2}_{L^{2}(Q)}\right),
\end{array}
\end{equation*}
proving \eqref{estimate for solution}.

Before continuing with the demonstration of this result, we note that the system \eqref{eq:linearized_system} has its first equation acting forward in time and its other two backward in time. Therefore, to establish estimates \eqref{des Proposition 5} and \eqref{des Proposition 6} we need to change variables in the form $\tilde{p}^{i}(x,t)=p^{i}(x,T-t)$ and $ \tilde{H}_{i}(x,t)=H_{i}(x,T-t)$, with $i=\lbrace 1,2\rbrace$, so that we obtain a system with new variables $(\tilde {y},\tilde{p}^{i},\tilde{H}_{i})$ in which all its equations are forward in time. In this way, the new system have initial conditions $(\tilde{y}(\cdot,0),\tilde{p}^{1}(\cdot,0),\tilde{p}^{2}(\cdot, 0 ))=(y_{0},0,0)$ and the estimates with weights for the solution, as well as regularity results, will be equivalent to those of the system \eqref{eq:linearized_system}. Keeping this in mind, for simplicity, we will maintain the \eqref{eq:linearized_system} notation for the then system
\begin{equation}
\label{eq:linearized_system new}
\left\{\begin{aligned}
&y_t - (\ell(0) a(x) y_x)_x = h \mathds{1}_{\mathcal{O}} -\frac{1}{\mu_1} p^1 \mathds{1}_{\mathcal{O}_{1}} - \frac{1}{\mu_2} p^2 \mathds{1}_{\mathcal{O}_{2}} + H &&\text{in}&& Q, \\
&p_t^i - \left(\ell(0) a(x) p_x^i \right)_x  = \alpha_i y \mathds{1}_{\mathcal{O}_{i,d}} + H_i  &&\text{in}&& Q,\\
&y(0,t)=y(1,t)=0, \ \ p^i(0,t) = p^i(1,t) = 0 &&\text{on}&&\ (0,T), \\
&y(\cdot,0) = y_{0}, \ \ p^i(\cdot,0) = 0 &&\text{in}&& (0,1),
\end{aligned}
\right.
\end{equation}
for $i=\lbrace 1,2\rbrace.$

Let us multiply $\eqref{eq:linearized_system new}_{1}$ by $\hat{\rho}^{2} y$, $\eqref{eq:linearized_system new}_{2}$ by $\hat{\rho}^{2} p^{i}$ and integrate over $[0,1]$. Hence, using that $\hat{\rho}^{2} = \rho_{0}\rho_{1}$, and $\rho_{1}\leq C\rho_{2}$, we compute
{
\begin{equation}\label{second estimate}
    \begin{array}{l}
     \dfrac{1}{2}\dfrac{d}{dt}\displaystyle\int_{0}^{1}\hat{\rho}^{2}|y|^{2}dx +\sum_{i=1}^{2}\dfrac{1}{2}\dfrac{d}{dt}\displaystyle\int_{0}^{1}\hat{\rho}^{2} | p^{i}|^{2}dx + \displaystyle\int_{0}^{1}\hat{\rho}^{2}\ell(0)a(x)|y_{x}|^{2} dx+ \sum_{i=1}^{2}\displaystyle\int_{0}^{1}\hat{\rho}^{2}\ell(0)a(x)|p_{x}^{i}|^{2}dx\hspace{0.1cm}\\
          \leq C\left(\displaystyle\int_{0}^{1}\rho_{0}^{2}|y|^{2}dx + \sum_{i=1}^{2}\displaystyle\int_{0}^{1}\rho_{0}^{2}|p^{i}|^{2}dx + \displaystyle\int_{\mathcal{O}}\rho_{1}^{2}|h|^{2}dx + \displaystyle\int_{0}^{1}\rho_{2}^{2}|H|^{2}dx+
          \sum_{i=1}^{2}\displaystyle\int_{0}^{1}\rho_{2}^{2}|H_{i}|^{2}dx\right)\hspace{0.1cm}\\
          \,\,\, + \ \mathcal{M},
    \end{array}
\end{equation}}
where $\mathcal{M} = \int_{0}^{1}\hat{\rho}(\hat{\rho})_{t}|y|^{2}dx +\sum_{i=1}^{2}\int_{0}^{1}\hat{\rho}(\hat{\rho})_{t}|p^{i}|^{2}dx$, and $(\cdot)_{t}=\frac{d}{dt}(\cdot)$. To facilitate notation, we will omit the sum. Recall that
$A^*(t) = C_1 \tau(t)$, and $\zeta^*(t) = C_2 \tau(t)$, then we have that $A^*_{t}=\bar C (\zeta^*)_{t}$ and consequently
\begin{equation*}
    \begin{split}
      \hat{\rho}(\hat{\rho})_{t} &= e^{-sA^*}(\zeta^*)^{-3}\left(-se^{-sA^*} A^*_{t}(\zeta^*)^{-3} -3 e^{-sA^*}(\zeta^*)^{-4}(\zeta^*)_{t}\right)    \\
       &= -e^{-2sA^*}(\zeta^*)^{-4}(\zeta^*)_{t}\left(s(\zeta^*)^{-2}\bar{C} + 3(\zeta^*)^{-3} \right)  \\
        &=-\rho_{0}^{2}(\zeta^*)_{t}\left(s(\zeta^*)^{-2}\bar{C} + 3(\zeta^*)^{-3} \right).
    \end{split}
\end{equation*}
Thus, for any $t\in [0,T)$,
\begin{equation*}
    \begin{array}{l}
        |\hat{\rho}(\hat{\rho})_{t}|\leq C\rho_{0}^{2}\tau^{2}|s(\zeta^*)^{-2}\bar{C} + 3(\zeta^*)^{-3}|  
         \leq C\rho_{0}^{2}|s\bar{C}+3(\zeta^*)^{-1}| \leq C\rho_{0}^{2},
    \end{array}
\end{equation*}
and we obtain
\begin{equation*}
    \mathcal{M}\leq C \displaystyle\int_{0}^{1}{\rho_{0}^{2}}(|y|^{2} +  |p^{i}|^{2})dx.
\end{equation*}
Therefore, \eqref{second estimate} becomes
\begin{equation*}
    \begin{array}{l}
       \dfrac{1}{2}\dfrac{d}{dt}\displaystyle\int_{0}^{1}\hat{\rho}^{2}(|y|^{2}+| p^{i}|^{2})dx + \dfrac{1}{2}\displaystyle\int_{0}^{1}\hat{\rho}^{2}\ell(0)a(x)(|y_{x}|^{2}+|p^{i}_{x}|^{2})dx\vspace{0.1cm}\\
          \leq C\left(\displaystyle\int_{0}^{1}\rho_{0}^{2}(|y|^{2} + |p^{i}|^{2})dx + \displaystyle\int_{\mathcal{O}}\rho_{1}^{2}|h|^{2}dx + \displaystyle\int_{0}^{1}\rho_{2}^{2}(|H|^{2}+|H_{i}|^{2})dx\right)  
    \end{array}
\end{equation*}
and, integrating over time, we conclude \eqref{des Proposition 5}.

Now, to prove \eqref{des Proposition 6}, multiply $\eqref{eq:linearized_system new}_{1}$ by $\rho_{1}^{2} y_{t}$ and $\eqref{eq:linearized_system new}_{2}$ by $\rho_{1}^{2} p^{i}_{t}$ and integrate over $[0,1]$. Thus, we get

\begin{equation}\label{third estimate}
    \begin{array}{l}
\displaystyle\int_{0}^{1}\rho_{1}^{2}(|y_{t}|^{2}+|p^{i}_{t}|^{2})dx + \dfrac{1}{2}\dfrac{d}{dt}\displaystyle\int_{0}^{1}\rho_{1}^{2}\ell(0)a(x)(|y_{x}|^{2}+|p^{i}_{x}|^{2})dx  \vspace{0.1cm}\\
         \leq C\left(\displaystyle\int_{0}^{1}\rho_{0}^{2}(|y|^{2} + |p^{i}|^{2})dx + \displaystyle\int_{\mathcal{O}}\rho_{1}^{2}|h|^{2}dx + \displaystyle\int_{0}^{1}\rho_{2}^{2}(|H|^{2}+|H_{i}|^{2})dx\right) \vspace{0.1cm}\\
        + \ \dfrac{1}{2}\displaystyle\int_{0}^{1}\rho_{1}^{2}(|y_{t}|^{2}+|p^{i}_{t}|^{2})dx + |\widetilde{\mathcal{M}}|,
    \end{array}
\end{equation}
where $\widetilde{\mathcal{M}}= \frac{1}{2}\int_{0}^{1}(\rho_{1}^{2})_{t}\,\ell(0)a(x)(|y_{x}|^{2}+|p^{i}_{x}|^{2})dx$. Since $|\zeta_{t}|\leq C\zeta^{2}$, we have that $|(\rho^{2}_{1})_{t}|\leq C\hat{\rho}^{2}$. Hence, $|\widetilde{\mathcal{M}}| \leq C\int_{0}^{1}\hat{\rho}^{2}\,\ell(0)a(x)(|y_{x}|^{2}+|p^{i}_{x}|^{2})dx$. 
 
Thus, \eqref{third estimate} gives 
\begin{equation*}
    \begin{array}{l}
\dfrac{1}{2}\displaystyle\int_{0}^{1}\rho_{1}^{2}(|y_{t}|^{2}+|p^{i}_{t}|^{2})dx + \dfrac{1}{2}\dfrac{d}{dt}\displaystyle\int_{0}^{1}\rho_{1}^{2}\ell(0)a(x)(|y_{x}|^{2}+|p^{i}_{x}|^{2})dx  \vspace{0.1cm}\\
         \leq C\left(\displaystyle\int_{0}^{1}\rho_{0}^{2}(|y|^{2} + |p^{i}|^{2})dx + \displaystyle\int_{\mathcal{O}}\rho_{1}^{2}|h|^{2}dx + \displaystyle\int_{0}^{1}\rho_{2}^{2}(|H|^{2}+|H_{i}|^{2})dx\right) \vspace{0.1cm}\\
        \, + \, C\displaystyle\int_{0}^{1}\hat{\rho}^{2}\,\ell(0)a(x)(|y_{x}|^{2}+|p^{i}_{x}|^{2})dx.
    \end{array}
\end{equation*}
 Integrating the previous inequality from $0$ to $t$ and using the estimate \eqref{des Proposition 5} we arrive at
 \begin{equation}\label{primeira da des proposição 6}
     \begin{array}{c}
\displaystyle\sup_{[0,T]}(\rho_{1}^{2}\|\sqrt{a}y_{x} \|^{2}_{L^{2}(0,1)})+ \displaystyle\sup_{[0,T]}(\rho_{1}^{2}\|\sqrt{a}p^{i}_{x} \|^{2}_{L^{2}(0,1)}) + \displaystyle\int_{Q}\rho_{1}^{2}(|y_{t}|^{2}+|p^{i}_{t}|^{2})dxdt\\
\leq C \kappa_{1}(H,H_{1},H_{2},y_{0}).
     \end{array}
 \end{equation}

Finally, to conclude \eqref{des Proposition 6}, it remains to estimate $\int_{Q}\rho_{1}^{2}(|(\ell(0)a(x)y_{x})_{x}|^{2} + |(\ell(0)a(x)p^{i}_{x})_{x}|^{2})dxdt$. To do this, it is enough to multiply $\eqref{eq:linearized_system new}_{1}$ by $-\rho_{1}^{2}(\ell(0)a(x)y_{x})_{x}$ and $\eqref{eq:linearized_system new}_{2}$ by $-\rho_{1}^{2}(\ell(0)a(x)p^{i}_{x})_{x}$, integrate over $[0,1]$ and continue with the same previous reasoning. Indeed, after multiplications and integration we deduce 
\begin{equation*}
    \begin{array}{l}
         \dfrac{1}{2}\dfrac{d}{dt}\displaystyle\int_{0}^{1}\rho_{1}^{2}\ell(0)\,a(x)(|y_{x}|^{2}+|p^{i}_{x}|^{2})dx + \dfrac{1}{2}\displaystyle\int_{0}^{1}\rho_{1}^{2}(|(\ell(0)a(x)y_{x})_{x}|^{2} + |(\ell(0)a(x)p^{i}_{x})_{x}|^{2} )dx\vspace{0.1cm}\\
         \leq C\left(\displaystyle\int_{0}^{1}\rho_{0}^{2}(|y|^{2} + |p^{i}|^{2})dx + \displaystyle\int_{\mathcal{O}}\rho_{1}^{2}|h|^{2}dx + \displaystyle\int_{0}^{1}\rho_{2}^{2}(|H|^{2}+|H_{i}|^{2})dx\right) \vspace{0.1cm}\\
         +\, C\displaystyle\int_{0}^{1}\hat{\rho}^{2}\,\ell(0)a(x)(|y_{x}|^{2}+|p^{i}_{x}|^{2})dx + C\displaystyle\int_{0}^{1}\rho_{1}^{2}(|y_{t}|^{2} + |p^{i}_{t}|^{2})dx.
    \end{array}
\end{equation*}
Therefore, integrating over time and using estimates \eqref{des Proposition 5} and \eqref{primeira da des proposição 6}  we obtain
\begin{equation}\label{segunda da des proposição 6}
    \begin{array}{l}
\displaystyle\int_{Q}\rho_{1}^{2}(|(\ell(0)a(x)y_{x})_{x}|^{2} + |(\ell(0)a(x)p^{i}_{x})_{x}|^{2})dxdt  \leq C \kappa_{1}(H,H_{1},H_{2},y_{0}).     \end{array}
\end{equation}
From \eqref{primeira da des proposição 6} and \eqref{segunda da des proposição 6} we infer \eqref{des Proposition 6}.

This ends the proof.
\end{proof}

\section{Local null controllability of the nonlinear system }\label{control for nonlinear system}
We are interested in using Liusternik's inverse function theorem (on the right) to obtain our result. This theorem can be found in \cite{Alekseev} and is given below, where $B_{r}(0)$ and $B_{\delta}(\zeta_{0})$ are open ball, respectively of radius $r$ and $\delta$. All along this section we use the weights defined in \eqref{eq:weights_rhos}.
\begin{theorem}[Liusternik]\label{Liusternik}
Let $ \mathcal{Y}$ and $ \mathcal{Z}$ be Banach spaces and let $\mathcal{A}:B_{r}(0)\subset  \mathcal{Y}\rightarrow  \mathcal{Z}$ be a $C^{1}$ mapping. Let as assume that $\mathcal{A}^{\prime}(0)$ is onto and let us set $\mathcal{A}(0)=\zeta_{0}$. Then, there exist $\delta >0$, a mapping $W: B_{\delta}(\zeta_{0})\subset  \mathcal{Z}\rightarrow  \mathcal{Y}$ and a constant $K>0$ such that
\begin{equation*}
    W(z)\in B_{r}(0),\,\, \mathcal{A}(W(z))=z\,\, \text{and}\,\, \Vert W(z)\Vert_{ \mathcal{Y}}\leq K\Vert z-\zeta_{0}\Vert_{ \mathcal{Z}}\, \, \forall\, z\in B_{\delta}(\zeta_{0}).
\end{equation*}
In particular, $W$ is a local inverse-to-the-right of $\mathcal{A}$.
\end{theorem}
In order to do so, we define a map $\mathcal{A} : \mathcal{Y} \to \mathcal{Z}$ between suitable Banach spaces $\mathcal{Y}$ and $\mathcal{Z}$ whose definition and properties came from the controllability result of the linearized system and the additional estimates shown in Theorem \ref{theorem case linear}.  

Denoting 
\begin{equation}
    \left\{\begin{aligned}
      &H=y_t - (\ell(0) a(x) y_x)_x - h \mathds{1}_{\mathcal{O}} + \frac{1}{\mu_1} p^1 \mathds{1}_{\mathcal{O}_{1}} + \frac{1}{\mu_2} p^2 \mathds{1}_{\mathcal{O}_{2}},\\
      &H_{i}= -p_t^i - \left(\ell(0) a(x) p_x^i \right)_x  - \alpha_i y \mathds{1}_{\mathcal{O}_{i,d}}, \quad i=\lbrace1,2\rbrace,\\
\end{aligned}\right.
\end{equation}
let us define the space
\begin{eqnarray}\label{eq:espaceY}
\mathcal{Y}&=&\{(y,p^1,p^2,h)\in [L^{2}(Q)]^{3}\times L^{2}(\mathcal{O}\times (0,T)); \ y(\cdot,t), p^{1}(\cdot,t), p^{2}(\cdot,t) \nonumber\\ 
&\quad&\text{are absolutely continuous in}\ [0, 1],\ \text{a.e. in}\ [0, T], \ \rho_{1}h\in L^{2}( \mathcal{O} \times (0,T)),\nonumber \\
 &\quad&\rho_{0}y, \rho_{0}p^{i}, \rho_2 H \in L^2(Q),
 \rho_2H_{i} \in L^2(Q),\, y(1, t)\equiv y(0,t)\equiv0, \, \nonumber\\
 &\quad&  p^{i}(1,t)\equiv p^{i}(0,t)\,\equiv 0 \ \text{a.e in}\ [0, T], \forall i\in\lbrace1,2\rbrace,\,\, y(\cdot,0) \in H_a^1(0,1) \}.
\end{eqnarray}
Thus, $\mathcal{Y}$ is a Hilbert space for the norm $\Vert \cdot\Vert_{\mathcal{Y}}$, where
\begin{eqnarray*}
          \Vert (y,p^1,p^2,h)\Vert^{2}_{\mathcal{Y}} &=& \Vert \rho_{0}y\Vert^{2}_{L^{2}(Q)} + \Vert \rho_{0}p^{i}\Vert^{2}_{L^{2}(Q)} + \Vert\rho_{1}h\Vert^{2}_{L^{2}(\mathcal{O}\times (0,T))} + \Vert\rho_{2} H\Vert^{2}_{L^{2}(Q)}\\
          &\quad &+\, \Vert\rho_{2} H_{i}\Vert^{2}_{L^{2}(Q)} + \Vert y(\cdot,0)\Vert^{2}_{H^{1}_{a}(0,1)}.  
\end{eqnarray*}

Now, let us introduce the Banach space $\mathcal{Z} = \mathcal{F} \times \mathcal{F} \times \mathcal{F} \times H_a^1(0,1)$ such that 
$$\mathcal{F}=\{ z \in L^2(Q) \ ; \ \rho_2 z \in L^2(Q) \}.$$
Finally, consider the map $\mathcal{A} : \mathcal{Y} \to \mathcal{Z}$ such that
$(y,p^1,p^2,h) \mapsto \mathcal{A}(y,p^1,p^2,h) = (\mathcal{A}_0, \mathcal{A}_1, \mathcal{A}_2, \mathcal{A}_3)$ where the components are given by
\begin{equation}\label{aplicação A}
\left\{\begin{array}{l}
\mathcal{A}_0(y,p^1,p^2,h) = y_t - \left( a(x)\ell\left(\displaystyle\int_0^1 y dx\right) y_x \right)_x - h \mathds{1}_{\mathcal{O}} + \dfrac{1}{\mu_1} p^1 \mathds{1}_{\mathcal{O}_{1}} + \dfrac{1}{\mu_2} p^2 \mathds{1}_{\mathcal{O}_2},\\
\mathcal{A}_i(y,p^1,p^2,h) = -p_t^i - \left( a(x) \ell\left(\displaystyle\int_0^1 y dx\right) p_x^i \right)_x  +  \ell'\left(\displaystyle\int_0^1 y dx\right) \displaystyle\int_0^1 ( a(x) y_x \ p^i_x) dx\\
\hspace{3.1cm}- \alpha_i (y-y_{i,d}) \mathds{1}_{\mathcal{O}_{i,d}},\,\,\text{for}\,\, i=\lbrace1, 2\rbrace, \\
\mathcal{A}_3(y,p^1,p^2,h) = y(\cdot,0).
\end{array}\right.
\end{equation}

\color{black}
\begin{remark}
    Observe that, based on the definition of the set $\mathcal{Y}$, it is clear that estimates \eqref{estimate for solution}, \eqref{des Proposition 5} and \eqref{des Proposition 6} in Theorem \ref{theorem case linear} are bounded by the norms of the elements in $\mathcal{Y}$.
\end{remark}

Our goal is to prove that $\mathcal{A}:\mathcal{Y}\to\mathcal{Z}$ satisfies the conditions of Theorem \ref{Liusternik}, namely, that $\mathcal{A}$ is well defined and continuous, continuously differentiable, and that $\mathcal{A}^{\prime}(0,0,0,0)$ is onto. To this end, we first present a fundamental result that will be crucial for our analysis.

\begin{proposition}\label{Corollary 2}
There exists $C>0$ such that 
    \begin{equation}\label{eq:coro_2_est1}
        \begin{array}{l}
\displaystyle\int_{Q}\rho_{2}^{2}\left(\displaystyle\int_{0}^{1}\bar y dx\right)^{2}|(a(x)y_{x})_{x}|^{2}dxdt    +    \displaystyle\int_{Q}\rho_{2}^{2}\left(\displaystyle\int_{0}^{1} \bar p^{1} dx\right)^{2}|(a(x)p^{1}_{x})_{x}|^{2}dxdt\\
+ \displaystyle\int_{Q}\rho_{2}^{2}\left(\displaystyle\int_{0}^{1}\bar p^{2} dx\right)^{2}|(a(x)p^{2}_{x})_{x}|^{2}dxdt\leq C\|(y,p^{1},p^{2},h)\|^{2}_{\mathcal{Y}} \ \|(\bar y, \bar p^{1},\bar p^{2}, \bar h)\|^{2}_{\mathcal{Y}},
        \end{array}
    \end{equation}

and    
\begin{equation}\label{eq:coro_2_est2}
        \begin{array}{l}
\displaystyle\int_{Q}\rho_{2}^{2}\left(\displaystyle\int_{0}^{1}\bar y dx\right)^{2} \Vert \sqrt{a(x)}y_{x}\Vert_{L^2(0,1)}^{2} dxdt    +    \displaystyle\int_{Q}\rho_{2}^{2}\left(\displaystyle\int_{0}^{1}\bar y dx\right)^{2} \Vert \sqrt{a(x)}p^1_{x}\Vert_{L^2(0,1)}^{2} dxdt\\ + \displaystyle\int_{Q}\rho_{2}^{2}\left(\displaystyle\int_{0}^{1}\bar y dx\right)^{2} \Vert \sqrt{a(x)}p^2_{x}\Vert_{L^2(0,1)}^{2} dxdt \leq C\|(y,p^{1},p^{2},h)\|^{2}_{\mathcal{Y}} \ \|(\bar y,\bar p^{1}, \bar p^{2},\bar h)\|^{2}_{\mathcal{Y}},
        \end{array}
\end{equation}    
for any $(y,p^{1},p^{2},h), (\bar y,\bar p^{1}, \bar p^{2}, \bar h)\in \mathcal{Y}$.   
\end{proposition}
\begin{proof}
    See Appendix \ref{appendix A}.
\end{proof}

Thus, we are in a position to prove that Theorem \ref{Liusternik} can be applied to the mapping $\mathcal{A}$ defined in \eqref{aplicação A}. We will verify this through the following three lemmas:

\begin{lemma}\label{A bem definido}
    Let $\mathcal{A}: \mathcal{Y}\rightarrow  \mathcal{Z}$ be given by \eqref{aplicação A}. Then, $\mathcal{A}$ is well defined and continuous. 
\end{lemma}
\begin{proof}
    We want to show that $\mathcal{A}(y,p^{1},p^{2},h)$ belongs to $ \mathcal{Z}$, for every $(y,p^{1},p^{2},h)\in  \mathcal{Y}$. We will therefore show that each $\mathcal{A}_{i}(y,p^{1},p^{2},h)$, with $i=\lbrace 0,1, 2, 3\rbrace$, defined in \eqref{aplicação A} belongs to its respective space. Note that, 
\begin{eqnarray*}
\|\mathcal{A}_{0}(y,p^{1},p^{2},h)\|^{2}_{\mathcal{F}}
&\leq& 2\int_{Q}\rho^{2}_{2}|H|^{2}dxdt  2\int_{Q}\rho^{2}_{2}\left|\left(\ell\left(\int_{0}^{1}ydx\right)a(x)y_{x}\right)_{x} -\left(\ell(0)a(x)y_{x}\right)_{x}\right|^{2}dxdt\\
             & = &  2 I_{1} + 2 I_{2}.    
\end{eqnarray*}
It is immediate by definition of the space $\mathcal{Y}$ that $I_{1}\leq C \|(y,p^{1},p^{2},h)\|^{2}_{\mathcal{Y}}$. Furthermore, since $\ell$ is Lipschitiz-continuous and applying Proposition \ref{Corollary 2}, inequality \eqref{eq:coro_2_est1}, we obtain
\begin{eqnarray*}
I_{2}=\int_{Q}\rho^{2}_{2}\left|\left[\ell\left(\int_{0}^{1}ydx\right) - \ell(0)\right](a(x)y_{x})_{x}\right|^{2}dxdt&\leq& C\int_{Q}\rho_{2}^{2}\left(\int_{0}^{1}ydx\right)^{2}|(a(x)y_{x})_{x}|^{2}dxdt\\
&\leq& C\|(y,p^{1},p^{2},h)\|^{4}_{\mathcal{Y}}.    
\end{eqnarray*}
Thus, $\mathcal{A}_{0}(y,p^{1},p^{2},h)\in \mathcal{F}.$

Now let us analyze $\mathcal{A}_{i}(y,p^{1},p^{2},h)$, with $i=\lbrace 1,2\rbrace$. Note the following:
\begin{eqnarray*}
\Vert\mathcal{A}_{i}(y,p^{1},p^{2},h)\Vert^{2}_{\mathcal{F}}
&\leq& 2\int_{Q}\rho^{2}_{2}|H_i|^{2}dxdt
+ 2\int_{Q}\rho^{2}_{2}\left| \left[ \ell\left(\int_{0}^{1}y\,dx\right) - \ell(0) \right] \left( a(x)p^i_{x}\right)_{x} \right|^{2}dxdt \\
&\quad& + \ 2\int_{Q}\rho^{2}_{2}\left| \ell'\left(\int_{0}^{1}y\,dx\right) \int_0^1 (a(x) y_x p^i_{x})dx \right|^{2}dxdt\\
&\quad&+
2\int_{Q}\rho^{2}_{2}\left| \alpha_i y_{i,d} \mathds{1}_{\mathcal{O}_{i,d}}\right|^{2}dxdt   =2I_{3} + 2I_{4}+2I_5+2I_6.   
\end{eqnarray*}
By definition of the space $\mathcal{Y}$ we have that $I_{3}\leq C \|(y,p^{1},p^{2},h)\|^{2}_{\mathcal{Y}}$. 
Since $\ell$ is Lipschitiz-continuous and applying inequality \eqref{eq:coro_2_est1} as above, we obtain $I_{4}\leq C \|(y,p^{1},p^{2},h)\|^{4}_{\mathcal{Y}}$. 
Moreover, by Proposition \ref{Corollary 2}, we obtain
\begin{eqnarray*}
I_5\leq C \int_Q \rho_2^2 \left| \int_0^1 (a(x) y_x p^i_x)dx \right|^2dxdt &\leq&  C\int_Q \rho_2^2  \left( \|\sqrt{a}y_x\|_{L^2(0,1)}^2 +\|\sqrt{a}p^i_x\|_{L^2(0,1)}^2\right)dxdt \\
&\leq& C \|(y,p^{1},p^{2},h)\|^{2}_{\mathcal{Y}}.    
\end{eqnarray*}
Finally, by the hypothesis $\rho^2 y_{i,d} \in L^2(Q)$, we have that $I_6$ is bounded. Thus $\mathcal{A}$ is well-defined.

\end{proof}
\begin{lemma}\label{DA continuo}
    The mapping $\mathcal{A}: \mathcal{Y}\longrightarrow  \mathcal{Z}$ is continuously differentiable.
\end{lemma}
\begin{proof}
First we prove that $\mathcal{A}$ is Gateaux differentiable at any $(y,p^1,p^2,h) \in \mathcal{Y}$ and let us compute the
$\textit{G-derivative}$ ${\mathcal{A}}^{\prime}(y, p^{1}, p{^2}, h)$.
Consider the linear mapping $D \mathcal{A}: \mathcal{Y} \to \mathcal{Z}$ given by
$$
D\mathcal{A}(y,p^1,p^2,h) = (D\mathcal{A}_0,D\mathcal{A}_1,D\mathcal{A}_2,D\mathcal{A}_3),
$$
where, for $i=\lbrace 1, 2 \rbrace$ and  $(\bar y, \bar p^1,\bar p^2,\bar h) \in \mathcal{Y}$,
\begin{equation}
\left\{\begin{aligned}
D\mathcal{A}_0(\bar y, \bar p^1,\bar p^2,\bar h) = &  \, \bar y_t - \ell'\left(\int_0^1 y \,dx\right) \int_0^1 \bar y \,dx (a(x)y_x)_x - \ell\left(\int_0^1 y\, dx\right) (a(x) \bar y_x)_x\\
& - \bar h \mathds{1}_\mathcal{O} + \frac{1}{\mu_1} \bar p^1 \mathds{1}_{\mathcal{O}_{1}} + \frac{1}{\mu_2} \bar p^2 \mathds{1}_{\mathcal{O}_{2}}, \\
D\mathcal{A}_i(\bar y, \bar p^1,\bar p^2,\bar h) =& - \bar p_t^i - \ell'\left(\int_0^1 y\, dx\right) \int_0^1 \bar y \,dx (a(x)p_x^i)_x - \ell\left(\int_0^1 y\, dx\right) (a(x) \bar p_x^i)_x \\
&+ \ell''\left(\int_0^1 y \,dx\right) \int_0^1 \bar y \,dx \left( \int_0^1 a(x) y_x p_x^i \,dx \right) \\
&+ \ell'\left(\int_0^1 y \,dx\right)\left(\int_0^1 a(x) \bar y_x p_x^i\, dx  + \int_0^1 a(x) y_x \bar p_x^i \,dx - \alpha_i \bar y \mathds{1}_{\mathcal{O}_{i,d}}\right), \\
D\mathcal{A}_3(\bar y, \bar p^1,\bar p^2,\bar h) =& \, \bar y(0).
\end{aligned}
\right.
\end{equation}
We have to show that, for $i=\lbrace 0, 1, 2, 3\rbrace$, 
$$
\frac{1}{\lambda}\left[ \mathcal{A}_i ((y,p^1,p^2,h)+\lambda(\bar y, \bar p^1,\bar p^2,\bar h)) - \mathcal{A}_i (y,p^1,p^2,h) \right] \rightarrow D\mathcal{A}_i(\bar y, \bar p^1,\bar p^2,\bar h),
$$
strongly in the corresponding factor of $\mathcal{Z}$ as $\lambda \to 0$.

Indeed, we have 
\begin{equation*}
\begin{split}
&\left\| \frac{1}{\lambda}\left[ \mathcal{A}_0 ((y,p^1,p^2,h)+\lambda(\bar y, \bar p^1,\bar p^2,\bar h)) - \mathcal{A}_0(\bar y, \bar p^1,\bar p^2,\bar h) \right] - D\mathcal{A}_0(\bar y, \bar p^1,\bar p^2,\bar h) \right\|^{2}_{\mathcal{F}} \\
&=\left\|
\bar y_t - \frac{1}{\lambda} \left(a(x) \ell\left(\int_0^1(y+\lambda\bar y)dx\right) (y_x + \lambda\bar y_x) - a(x)\ell\left(\int_0^1 y\,dx\right) y_x ) \right)_x \right. \\
&\left.\,\, -\bar h \mathds{1}_\mathcal{O} + \frac{1}{\mu_1} \bar p^1 \mathds{1}_{\mathcal{O}_{1}} + \frac{1}{\mu_2} \bar p^2 \mathds{1}_{\mathcal{O}_{2}} - D\mathcal{A}_0(\bar y, \bar p^1,\bar p^2,\bar h) \right\|^{2}_{\mathcal{F}} \\
&= \left\| - \frac{1}{\lambda} \left(a(x) \ell\left(\int_0^1(y+\lambda\bar y)dx\right) (y_x + \lambda\bar y_x) - a(x)\ell\left(\int_0^1 ydx\right) y_x  \right)_x\right. \\
& \ + \ell'\left(\int_0^1 y \,dx\right) \int_0^1 \bar y dx (a(x)y_x)_x ;
\left.\,\,  + \,\ell\left(\int_0^1 y \,dx\right) (a(x) \bar y_x)_x \right\|^{2}_{\mathcal{F}}
\end{split}
\end{equation*}
\begin{equation*}
\begin{split}
&\leq \left\| 
\left( a(x) \frac{1}{\lambda} \left( \ell\left(\int_0^1 (y+\lambda\bar y)dx\right) - \ell\left(\int_0^1 y\,dx\right)\right) y_x \right)_x - \left(a(x)\ell'\left(\int_0^1 y\,dx\right) \left( \int_0^1 \bar y\,dx\right) y_x \right)_x
\right\|^{2}_{\mathcal{F}} \\
&\,\,+ \left\| \ell\left(\int_0^1 (y+\lambda\bar y)dx\right) (a(x)\bar y_x)_x - \ell\left(\int_0^1 y\,dx\right)(a(x)\bar y_x)_x \right\|^{2}_{\mathcal{F}} \\
& \leq \int_Q \rho_2^2 \left| \frac{1}{\lambda} \left( \ell\left(\int_0^1 (y+\lambda\bar y)dx\right) - \ell\left(\int_0^1 y\,dx\right) \right) - \ell'\left(\int_0^1 y\,dx\right)\int_0^1 \bar y\, dx \right|^2 |(a(x)y_x)_x|^2 dxdt\\
&\,\, + \int_Q \rho_2^2 \left| \ell\left(\int_0^1 (y+\lambda\bar y)dx\right) - \ell\left(\int_0^1 y\,dx\right) \right|^2 |(a(x)\bar y_x)_x|^2 dxdt = J_1 + J_2.
\end{split}
\end{equation*}
By the mean value theorem, there exists $c(t) \in [\bar{a}(t),\bar{b}(t)]$, where $\bar{a}(t)=\min \{ \int_0^1 ydx , \int_0^1 (y +\lambda \bar y)dx \}$ and $\bar{b}(t) = \max \{ \int_0^1 ydx , \int_0^1 (y +\lambda \bar y)dx \}$, such that
$$J_1 = \int_Q \rho_2^2 \left| \ell'(c) \int_0^1 \bar y dx- \ell'\left(\int_0^1 y\,dx\right)\int_0^1 \bar y\, dx\right|^2 |(a(x)y_x)_x|^2 dxdt.$$
Thus, since $\ell$ has bounded derivatives, $J_1 \leq C \int_Q \rho_2^2 |\int_0^1 \bar y dx|^2 |c-\int_0^1 y dx|^2 |(a(x)y_x)_x|^2dxdt$. By Proposition \ref{Corollary 2} and by the dominated convergence theorem, we get that $J_1$ converges to zero as $\lambda \to 0$. 

Moreover, $J_2 \leq C \int_Q \rho_2^2 | \lambda \int_0^1 \bar y dx|^2 |(a(x)y_x)_x|^2dxdt$. Again by Proposition \ref{Corollary 2} and the dominated convergence theorem, we get that $J_2$ converges to zero as $\lambda \to 0$. 

The case $i=3$ is immediate. Now, for $i=\lbrace 1, 2\rbrace$, we have
\begin{equation}
\begin{split}
&\left\| \frac{1}{\lambda}\left[ \mathcal{A}_i ((y,p^1,p^2,h)+\lambda(\bar y, \bar p^1,\bar p^2,\bar h)) - \mathcal{A}_i (y,p^1,p^2,h) \right] - D\mathcal{A}_i(\bar y, \bar p^1,\bar p^2,\bar h) \right\|^{2}_{\mathcal{F}}\\
&=\left\|
\bar y_t - \frac{1}{\lambda} \left[ \left( a(x) \ell\left(\int_0^1(y+\lambda\bar y)dx\right) (p^i_x + \lambda\bar p^i_x) - a(x)\ell\left(\int_0^1 y\,dx\right) p^i_x  \right)_x \right. \right. \\
&\left. \,+ \ell'\left(\int_0^1 (y +\lambda \bar y)dx\right)\left(\int_0^1 a(x)(y_x+\lambda\bar y_x)(p^i_x+\lambda \bar p^i_x)dx \right) - \ell'\left(\int_0^1 y\,dx\right)\left(\int_0^1 a(x) y_x p^i_x dx \right)  \right.
\\
&\left.\, -\alpha_i \lambda \bar y \mathds{1}_{\mathcal{O}_{i,d}} \right] - \left. D\mathcal{A}_i(\bar{y}, \bar{p}^1,\bar{p}^2,\bar{h}) \right \|^{2}_{\mathcal{F}} \\
&= \left\| \frac{1}{\lambda} \left( a(x) \ell\left(\int_0^1(y+\lambda\bar y)dx\right) (p^i_x + \lambda\bar p^i_x) - a(x)\ell\left(\int_0^1 y\,dx\right) p^i_x  \right)_x \right.\\
& \left. \, - \ell'\left(\int_0^1 y \,dx\right) \int_0^1 \bar y dx (a(x)p_x^i)_x - \ell\left(\int_0^1 y \,dx\right) (a(x) \bar p_x^i)_x \right\|^{2}_{\mathcal{F}} \\
&\,+ \left\| \frac{1}{\lambda} \left[ \ell'\left(\int_0^1 (y +\lambda \bar y)dx\right)\left(\int_0^1 a(x)(y_x+\lambda\bar y_x)(p^i_x+\lambda \bar p^i_x)dx \right) - \ell'\left(\int_0^1 y\,dx\right)\left(\int_0^1 a(x) y_x p^i_x dx\right) \right] \right. \\
&\left. \, - \ell''\left(\int_0^1 y \,dx\right) \int_0^1 \bar y\, dx \left( \int_0^1 a(x) y_x p_x^i dx \right) - \ell'\left(\int_0^1 y\, dx\right) \int_0^1 a(x) \bar y_x p_x^i dx \right. \\
&\left. \, - \ell'\left(\int_0^1 y \,dx\right) \int_0^1 a(x) y_x \bar p_x^i dx \right\|^{2}_{\mathcal{F}} := J_3 + J_4.
\end{split}
\end{equation}
First, we have
\begin{equation*}
\begin{split}
J_3 &= \int_Q \rho_2^2 \left| \frac{1}{\lambda} \left( \ell\left(\int_0^1 (y +\lambda \bar y)dx\right) - \ell\left(\int_0^1 y\,dx\right) \right) (a(x)p^i_x)_x + \ell\left(\int_0^1 (y +\lambda \bar y)dx\right) (a(x)\bar p^i_x)_x \right. \\
&\left. \,\, - \ell'\left(\int_0^1 y \,dx\right) \int_0^1 \bar y \,dx (a(x)p_x^i)_x - \ell\left(\int_0^1 y \, dx\right) (a(x) \bar p_x^i)_x \right|^2 dxdt\\
&\leq C \int_Q \rho_2^2 \left| \frac{1}{\lambda} \left( \ell\left(\int_0^1 (y +\lambda \bar y)dx\right) - \ell\left(\int_0^1 y\,dx\right) \right) - \ell'\left(\int_0^1 y dx\right) \int_0^1 \bar y dx \right|^2 |(a(x)p^i_x)_x|^2dxdt \\
&\,\,+ C \int_Q \rho_2^2 \left| \ell\left(\int_0^1 (y +\lambda \bar y)dx\right) - \ell\left(\int_0^1 y\,dx\right) \right|^2 |(a(x)\bar p^i_x)_x|^2dxdt,
\end{split}
\end{equation*}
and by the same estimates used for $J_1$ and $J_2$, Proposition \ref{Corollary 2} and by the dominated convergence theorem, we get that $J_3$ converges to zero as $\lambda \to 0$.

Moreover,
\begin{equation*}
\begin{split}
J_4 \leq& \int_{Q} \rho_2^2 \left| \frac{1}{\lambda} \left( \ell'\left(\int_0^1 (y +\lambda \bar y)dx\right) - \ell'\left(\int_0^1 y\,dx\right) \right) - \ell''\left(\int_0^1 y\, dx\right) \int_0^1 \bar y \,dx \right|^2 \left| \int_0^1 a(x) y_x p^i_x\,dx \right|^2 dxdt \\
&+ \int_Q \rho_2^2 \left| \ell'\left(\int_0^1 (y +\lambda \bar y)dx\right) - \ell'\left(\int_0^1 y\,dx\right) \right|^2 \left|\int_0^1 a(x) \bar y_x p^i_x \,dx\right|^2 dxdt \\
&+ \int_Q \rho_2^2 \left| \ell'\left(\int_0^1 (y +\lambda \bar y)dx\right) - \ell'\left(\int_0^1 y\,dx\right) \right|^2 \left|\int_0^1 a(x) y_x \bar p^i_x \,dx\right|^2 dxdt \\
&+ \lambda^2 \int_Q \rho_2^2 \left| \ell'\left(\int_0^1 (y +\lambda \bar y)dx\right) \right|^2 \left|\int_0^1 a(x) \bar y_x \bar p^i_x \,dx\right|^2 dxdt = K_1 + K_2 + K_3 + K4. 
\end{split}
\end{equation*}
Again by the mean value theorem there exists $c(t)$ such that
$$
K_1 \leq C \int_Q \rho_2^2 \left(\int_0^1 \bar y\,dx \right)^2 \left| c - \int_0^1 y dx \right|^2 \left| \int_0^1 a(x) y_x p^i_x\,dx \right|^2 dxdt.
$$
By Hölder inequality, 
$
K_1 \leq \int_Q \rho_2^2 \left(\int_0^1 \bar y\,dx \right)^2 \left| c - \int_0^1 y dx \right|^2 \left( \|\sqrt{a}y_x\|_{L^2(0,1)}^2 + \|\sqrt{a} p^i_x\|_{L^2(0,1)}^2 \right)
$, and using Proposition \ref{Corollary 2}, estimate \eqref{eq:coro_2_est2},  and the dominated convergence theorem, we have that $K_1 \to 0$ as $\lambda\to 0$.

Analogously, again by estimate \eqref{eq:coro_2_est2} and the dominated convergence theorem, we have that, for $j=\lbrace 2, 3, 4\rbrace$, $K_j \to 0$  as $\lambda\to 0$.
This finish the proof that $\mathcal{A}$ is Gateaux differentiable, with a \textit{G-derivative} $\mathcal{A^{\prime}}(y,{p^{1}},p^{2},h)= D\mathcal{A}(y,{p^{1}},p^{2},h)$.

To conclude the lemma, it is necessary to show that the mapping $ (y,p^{1},p^{2},h) \longmapsto \mathcal{A}'(y,p^{1},p^{2},h)$
is continuous from $\mathcal{Y}$ into $\mathcal{L}(\mathcal{Y},\mathcal{Z})$. Consequently, in light of classical results, it follows that $\mathcal{A}$ is Fr\'echet differentiable, and therefore of class $C^{1}$, which completes the proof of the lemma. The full proof, with all details, can be found in Appendix \ref{appendix B}.
\end{proof}

\begin{lemma}\label{Mapa sobrejetivo}
Let $\mathcal{A}$ be the mapping in \eqref{aplicação A}. Then, $\mathcal{A}^{\prime}(0,0,0,0)$ is onto.
\end{lemma}
\begin{proof}
Let $(H, H_{1}, H_{2}, y_{0})\in \mathcal{Z}$. From Theorem \ref{theorem case linear} we know there exists $y, p^{1}, p^{2}$ satisfying \eqref{eq:linearized_system} and \eqref{estimate for solution}. Furthermore, we know that $y, p^{1}, p^{2}\in C^{0}([0,T];L^{2}(0,1))\cap L^{2}(0,T;H^{1}_{a})$. Consequently, $(y,p^{1},p^{2},h)\in \mathcal{Y}$ and $$\mathcal{A}^{\prime}(0,0,0,0)\cdot(y,p^{1},p^{2},h)=(H,H_{1}, H_{2},y_{0}).$$ 
\hfill

This ends the proof.
\end{proof}
\noindent\textbf{Proof of Theorem \ref{thm:local_null_controllability}.}

\noindent According to Lemmas \ref{A bem definido}--\ref{Mapa sobrejetivo} we can apply the Inverse Mapping Theorem (Theorem \ref{Liusternik}) and consequently there exists $\delta > 0$ and a mapping $W:B_{\delta}(0)\subset {\mathcal{Z}}\rightarrow {\mathcal{Y}}$ such that
\begin{equation*}
    W(z)\in B_{r}(0)\,\,\, \text{and}\,\,\, {\mathcal{A}}(W(z))=z, \,\,\, \forall z\in B_{\delta}(0).
\end{equation*}
Taking $(0,0,0,y_{0})\in B_{\delta}(0)$ and $(y,p^{1},p^{2},h)=W(0,0,0,y_{0})\in {\mathcal{Y}}$, we have
\begin{equation*}
    {\mathcal{A}}(y,p^{1},p^{2},h)=(0,0,0,y_{0}).
\end{equation*}
Thus, we conclude that \eqref{eq:optimality_system} is locally null controllable at time $T > 0$.

\section{Additional comments and open questions}\label{sec:final_remarks}
\subsection{Controllability of a radial equation}
Let $\Omega \subset \R^n$, $n \geq 1$, be a bounded radial domain (euclidian ball) with boundary $\partial \Omega$. Let us denote $Q = \Omega \times (0,T)$ with lateral boundary $\Sigma = \partial \Omega \times (0,T)$. We are interested in the null controllability of a degenerate, nonlocal parabolic radial equation generalizing our one-dimensional result.  We start considering the following equation generalizing \eqref{eq:PDE}.
\begin{equation}
\label{eq:PDEradial}
\left\{\begin{aligned}
&y_t - \nabla \cdot \left(a(|x|) \ell\left(\int_{\Omega} y dx\right)  \nabla y \right) = h \mathds{1}_{\mathcal{O}} + v^1 \mathds{1}_{\mathcal{O}_{1}} + v^2 \mathds{1}_{\mathcal{O}_{2}} & \text{in } Q, \\
&y=0 & \text{on } \Sigma, \\
&y(\cdot,0) = y_{0} & \text{in } \Omega,
\end{aligned}
\right.
\end{equation}
with radial initial condition $y_{0}$, radial controls $h$ and $v^i$, $i=1,2$, and radially symmetric annular control domains $\mathcal{O}$, $\mathcal{O}_{1}$ and $\mathcal{O}_{2}$.

Using the radial symmetry we get an equation applied to the function  $z(r,t) = y(x,t)$, where $r=|x|$. In this radial case the $n$-dimensional Laplacian is given by $\Delta z = \frac{1}{r^{n-1}} (r^{n-1} z_r)_{r}$, with radial controls $h_R(r,t)=h(x,t)$, $v_R^i(r,t)=v^i(x,t)$, for $i=1,2$. Then, the left-hand side of equation \eqref{eq:PDEradial}$_1$ becomes
$y_t - \ell(\int_\Omega y\, dx) (a(r)  z_{rr} + a'(r) z_r + \frac{n-1}{r} a(r) z_r )$ and \eqref{eq:PDEradial} can be written as
\begin{equation*}
\left\{\begin{aligned}
&y_t - \ell\left(\int_\Omega y \,dx\right) (a(r) z_r)_r - \ell\left(\int_\Omega y\, dx\right) a(r) \frac{n-1}{r} z_r  = h_R \mathds{1}_{\mathcal{O}} + v_R^1 \mathds{1}_{\mathcal{O}_{1}} + v_R^2 \mathds{1}_{\mathcal{O}_{2}} & \text{in } Q, \\
&y=0 & \text{on } \Sigma, \\
&y(\cdot,0) = y_{0} & \text{in } \Omega.
\end{aligned}
\right.
\end{equation*}
Thus, our Theorem \ref{thm:local_null_controllability} can be applied to show the null controllability of the radially symmetric problem:
\begin{equation}
\label{eq:PDEradial2}
\left\{\begin{aligned}
&y_t - \nabla \cdot \left(a(|x|) \ell\left(\int_\Omega y\, dx\right)  \nabla y \right)\\ 
&+ \ell\left(\int_\Omega y\, dx\right) a(|x|) \nabla y \cdot \frac{(n-1)x}{|x|^2} = h \mathds{1}_{\mathcal{O}} + v^1 \mathds{1}_{\mathcal{O}_{1}} + v^2 \mathds{1}_{\mathcal{O}_{2}} & \text{in } Q, \\
&y(x,t)=0 & \text{on } \Sigma, \\
&y(\cdot,0) = y_{0} & \text{in } \Omega.
\end{aligned}
\right.
\end{equation}

\subsection{The strong degenerate case and boundary problem}\label{Subsection strong degenerate}
As described in the introduction, the assumptions made on the function $a$ lead us to the so-called weakly degenerate case. Consequently, Dirichlet boundary conditions are appropriate for the problem considered in \eqref{eq:PDE}. However, the constraints on $K$ can alter the behavior of $a$ and, therefore, modify the boundary conditions. In this sense, we can have another characterization, called the \textit{strongly degenerate condition}, defined in \cite{Alabau_cannarsa_fragnelli-06}, i.e., the function $a\in C^{1}([0,1])$ satisfies \eqref{conditionunder-a(x)} and \eqref{K-estimate-a(x)} with $K\in [1,2)$, where
\begin{equation*}\label{eq:theta_conditions}
\left\{
\begin{aligned}
&\exists\, \theta \in (1,K] \ \text{such that} \ \theta a \le x a' \ \text{near zero, if } K>1,\\[4pt]
&\exists\, \theta \in (0,1) \ \text{such that} \ \theta a \le x a' \ \text{near zero, if } K=1.
\end{aligned}
\right.
\end{equation*}
In that case, the natural boundary condition imposed at $x = 0$ for problem \eqref{eq:PDE} would be of Neumann type. In particular, suppose given $K\in [0,1)\cup [1,2)$ and $a$ satisfies \textit{strongly and weakly degenerate condition} in appropriate part, we can encompass both cases in the following system:
\begin{equation}\label{eq:systemcomplete-WS}
\left\{
\begin{aligned}
&y_{t} -  \left(a(x)\ell\!\left( \int_{0}^{1} y \, dx \right)y_x\right)_x  = h\mathds{1}_{\mathcal{O}}+v^{1}\mathds{1}_{\mathcal{O}_{1}}+v^{2}\mathds{1}_{\mathcal{O}_{2}} 
&&\text{in}&&  Q, \\
&y(1,t) = 0 \ \ \ \text{and}\ \ \ \left\{
\begin{aligned}
&y(0,t) = 0, && \text{(Weak)} \\
&\text{or}\\
&(a y_x)(0,t) = 0, && \text{(Strong)}
\end{aligned}
\right. &&\text{on}&&  (0,T), \\
&y(\cdot,0) = y_{0}(x), &&\text{in}&& (0,1).
\end{aligned}
\right.
\end{equation}
In context of hierarchical control to system \eqref{eq:systemcomplete-WS}, we can mention the results obtained in \cite{araruna2018stackelberg} where the authors treated the case $a(x)=x^{\gamma}$ with $\gamma\in[0,2)$ and $\ell\equiv1$. In particular, it is believed that analogous results could also be established for the strongly degenerate case for \eqref{eq:PDE}.

An interesting question that arises after considering the influence of the parameter $K$ on the boundary is when the controls can act on it. In this regard, it is necessary to carefully examine the influence of the parameter and identify on which part of the boundary the controls should be located. Indeed, suppose $x=0$, $\ell\equiv 1$, and $v^{1}=0=v^{2}$. Then, if the control $h$ acts at the boundary and $a$ satisfies the \textit{weakly degenerate condition} the system is controllable. However, if $a$ satisfies the \textit{strongly degenerate condition}, in particular when the control $h$ is located at the singular part of the boundary, even a certain kind of unique continuation is not known to hold; see, for instance, \cite{gueye2014exact}. On the other hand, when the control $h$ acts at the point $x=1$, we have that $a$ is not degenerate. In this case, we consider an extended system defined in $(0,1+\delta)\times(0,T)$ and take $\mathcal{O}\subset(1,1+\delta)$ so that the control $h$ is the solution of the extended system restricted to $x = 1$; see, for instance, \cite{Alabau_cannarsa_fragnelli-06}. From the point of view of hierarchical boundary controllability, all these questions remain open. This motivates some interesting open questions about the hierarchical boundary controls. 

Let $\ell:\mathbb{R}\to\mathbb{R}$ and $a$ satisfying the \text{weakly (or strongly) degenerate condition}. Consider the distributed and boundary controls system:
\begin{equation}\label{eq:boundary1}
	\left\{\begin{aligned}
		&y_t - \left(\ell(g)a(x)y_x\right)_x +F(y) =v^{1}\mathds{1}_{\mathcal{O}_{1}}+v^{2}\mathds{1}_{\mathcal{O}_{2}}  &&\text{in}&& Q,\\ 
        &y(0,t)=0, \ \ y(1,t)=h(t)    &&\text{on}&& (0,T),\\ 
		&y(\cdot,0) = y_0 &&\text{in}&& (0,1),
	\end{aligned}
	\right.
\end{equation}
and now suppose that $a$ satisfies the \text{weakly  degenerate condition} 
\begin{equation}\label{eq:boundary2}
	\left\{\begin{aligned}
		&y_t - \left(\ell(g)a(x)y_x\right)_x+F(y) =h\mathds{1}_{\mathcal{O}}  &&\text{in}&& Q,\\ 
        &y(0,t)=v^{1}(t), \ \ y(1,t)= v^{2}(t)    &&\text{on}&& (0,T),\\ 
		&y(\cdot,0) = y_0 &&\text{in}&& (0,1).
	\end{aligned}
	\right.
\end{equation}
where $F:\mathbb{R}\to\mathbb{R}$ is a  locally
Lipschitz-continuous function, $g$ is the nonlocal term, $h$ and $v^{1},v^{2}$ are controls leader and followers respectively. This problems was introduced in \cite{araruna2020hierarchical} when $\ell(g)$ and $a(x)=1$, in context of hierarchical control, applying the
Stackelberg-Nash strategy for parabolic system.

Another interesting question is to consider the system \eqref{eq:PDE} with only one follower, say $v^{1}$, and reverse the roles of the leader and the follower, following the ideas of \cite{calsavara2022new}. More precisely, for each leader control $h$, we associate the unique solution $v^{1}(h)$ of an extremal problem related to a cost functional that depends on the follower, such that the corresponding state $y$, associated with $h$ and $v^{1}(h)$, satisfies the terminal condition $y(\cdot, T) = 0$ in the interval $(0,1)$. Subsequently, we seek an admissible control $\hat{h}$ that minimizes a suitable cost functional associated with the leader's control. This approach makes the problem more challenging due to the class constraints imposed on the leader control. For other variations of this approach applied to the proposed problem, see Section 5 of \cite{calsavara2022new}.

\subsection{Hierarchical controllability of other degenerate systems}

We are interested in the hierarchical controllability of the following parabolic systems:
\begin{equation}
\label{eq:PDEopen1}
\left\{\begin{aligned}
&y_t - \mathcal{B}y = h\mathds{1}_{\mathcal{O}}+v^{1}\mathds{1}_{\mathcal{O}_{1}}+v^{2}\mathds{1}_{\mathcal{O}_{2}} &&\text{in}&& Q, \\
&y(0,t)=y(1,t)=0 &&\text{on}&& (0,T), \\
&y(\cdot,0) = y_{0} &&\text{in}&& (0,1),
\end{aligned}
\right.
\end{equation}
where the operator $\mathcal{B}$ can be considered in the following cases: 
\begin{itemize}
    \item[i)]\textit{Degenerate parabolic operator with  nonlocal coefficient - semi-linear - case:} 
    
    \begin{equation*}
    \mathcal{B}y=\left(a(x)\ell\left(\int_{0}^{1}y\ dx\right)y_{x}\right)_{x}+F(y)    
    \end{equation*}
   where $\ell:\mathbb{R}\to\mathbb{R}$ with $\ell(0)>0$ and $F:\mathbb{R}\to\mathbb{R}$ is a  locally
Lipschitz-continuous function;
    \item[ii)]\textit{Degenerate quasi-linear parabolic operator:}
    \begin{equation*}
    \mathcal{B}y=\nabla \cdot (a(\beta) b(y) \nabla y),    
    \end{equation*}
where $b=b(r)$ is a real function of class $C^3$ such that
\begin{equation*}
0  < b_0 \leq b(r) \leq b_1 \ \ \text{and} \ \  |b'(r)|+|b''(r)|+|b'''(r)| \leq M,\ \  \text{for any} \ \ r \in \mathbb{R}.    
\end{equation*}
Here, one can analyse the case  $\beta=x$ and the radial case $\beta=|x|$ with radial $y_{0}$.      
\end{itemize}

\color{black}

\section{Appendix}
\appendix
\renewcommand{\theequation}{\thesection.\arabic{equation}}
\section{- Proof of Proposition \ref{Corollary 2}}\label{appendix A}
We state the following lemma, which will be fundamental for the demonstration of Proposition \ref{Corollary 2}.
\begin{lemma}\label{lema para proposition}
    Define $\nu(x)=e^{\lambda(|\Psi|_\infty + \Psi)}-e^{3\lambda|\Psi|_\infty}$ and $\bar{\nu}=\displaystyle\max_{[0,1]}\nu(x)$. There exists $s>0$ such that, if $s\bar{\nu}<M<0$ then
    \begin{equation*}
        \begin{array}{l}
\displaystyle\sup_{t\in[0,T]}\left\{ e^{\frac{-2M}{m(t)}}\left[\left(\displaystyle\int_{0}^{1}y\,dx\right)^{2}+\left(\displaystyle\int_{0}^{1}p^{1}\,dx\right)^{2}+\left(\displaystyle\int_{0}^{1}p^{2}\,dx\right)^{2}
\right] \right\}\leq C\|(y,p^{1},p^{2},h)\|^{2}_{\mathcal{Y}},
        \end{array}
    \end{equation*}
    for all $(y,p^{1},p^{2},h)\in\mathcal{Y}$.
\end{lemma}
\begin{proof}
Following the ideas of \cite{DemarqueLimacoViana_deg_eq2018,DemarqueLimacoViana_deg_sys2020}, for each $(y,p^1,p^2,h) \in \mathcal{Y}$, consider, for $i=0,1,2$, the functions $q_i:[0,T]\to \R$, given by
$$
q_0(t) = e^{-\frac{M}{m(t)}} \int_0^1 y\, dx \qquad \text{and} \qquad q_i(t) = e^{-\frac{M}{m(t)}} \int_0^1 p^i\, dx,
$$
where $m(t)$ was defined in \eqref{eq:def_m}. Taking $k>0$, we have that $e^{-\frac{k}{m(t)}} \leq 8! \left(\frac{m(t)}{k}\right)^8$, for any $t \in [0,T]$. Since $A=\frac{\nu(x)}{m(t)}$, taking $s>0$ such that $2s(\bar \nu - \nu)>k$, we have
$$
-\frac{2M}{m(t)} + 2sA = - \frac{2M}{m(t)} + \frac{2s \nu}{m(t)} < \frac{-2s(\bar \nu - \nu)}{m(t)} < -\frac{k}{m(t)}. 
$$
Thus, $e^{-\frac{2M}{m(t)}} \leq e^{-2sA}e^{-\frac{k}{m(t)}} \leq C e^{-2sA} m(t)^8 \leq C_1 e^{-2sA} \zeta^{-8}  \leq C_2 \rho_1^2$, for any $t \in [0,T]$.
As a consequence, we claim that $\|q_i\|_{H^1(0,T)} \leq C\|(y,p^1,p^2,h)\|_\mathcal{Y}$. In fact, using that $\rho_1^2 \leq C \rho_0^2$,
\begin{eqnarray*}
\|q_0\|^{2}_{L^2(0,T)} &\leq& \displaystyle\int_0^T \int_0^1 e^{-\frac{2M}{m(t)}} |y|^2 dxdt\leq C\int_0^T \int_0^1 \rho_1^2 |y|^2 dxdt\leq C\displaystyle\int_0^T \int_0^1 \rho_0^2 |y|^2dxdt\\ 
&\leq& C \|(y,p^1,p^2,h)\|^{2}_\mathcal{Y}.    
\end{eqnarray*}
On the other hand,
\begin{eqnarray*}
\|q'_0\|^{2}_{L^2(0,T)} &\leq& \int_0^T \left| e^{-\frac{M}{m(t)}} \frac{M m'(t)}{m(t)^2} \int_0^1 y\,dx + e^{-\frac{M}{m(t)}} \int_0^1 y_t\,dx \right|^2 dt  \\
&\leq& C\int_0^T |q_0|^{2} \frac{M^2}{m^4(t)}dt + C \int_0^T \int_0^1 e^{-\frac{2M}{m(t)}} |y_t|^2 dxdt= K_1 + K_2.    
\end{eqnarray*}
Using that $M > s \bar \nu \geq s\nu = s A m(t)$ we have that $e^{-\frac{M}{m(t)}} \leq e^{-s A}$. Thus, 
$$K_2 \leq C \int_0^T \int_0^1 e^{-2sA} |y_t|^2 dxdt\leq C \int_0^T \int_0^1 \rho_1^2 |y_t|^2dxdt.$$
For $K_1$ we have
\begin{eqnarray*}
K_1 &\leq&C \displaystyle\int_0^T e^{-\frac{2M}{m(t)}} \int_0^1 |y|^2dx \frac{M^2}{m(t)^4}dt 
\leq C_1 \displaystyle\int_0^T \int_0^1 e^{-\frac{2M}{m(t)}}\frac{1}{m(t)^4} |y|^2dxdt\\ &=& C_2 \int_0^T \int_0^1 e^{-2sA}\zeta^{-4} |y|^2 dxdt.    
\end{eqnarray*}
Since $e^{-2sA}\zeta^{-4} = \rho_0^2 \leq C \rho_2^2$, we get $K_1 \leq C_3 \int_0^T \int_0^1 \rho_2^2 |y|^2dxdt$ and, $\|q_0'\|_{L^2(0,T)} \leq C \|(y,p^1,p^2,h)\|_\mathcal{Y}$. Thus proves the claim.
By analogous estimates for $q_1$ and $q_2$ we get
$$
\|q_0\|_{H^1(0,T)} + \|q_1\|_{H^1(0,T)} + \|q_2\|_{H^1(0,T)} \leq  C \|(y,p^1,p^2,h)\|_\mathcal{Y},
$$
and the conclusion follows from the immersion $H^1(0,T) \subset L^\infty(0,T)$.
\end{proof}

\begin{proof}[Proof of Proposition \ref{Corollary 2}.]
Take $(y,p^1,p^2,h)$, $(\bar y, \bar p^1, \bar p^2, \bar h) \in \mathcal{Y}$ and let $M < 0$ be the constant in Lemma \ref{lema para proposition}. Since $\rho_2^2 \rho_1^{-2} = \zeta^6$, we have
$$
\int_Q \rho_2^2 \left(\int_0^1 \bar y\,dx \right)^2 |(a(x)y_x)_x|^2 dxdt = \int_Q \zeta^6 \left(\int_0^1 \bar y\,dx \right)^2 \rho_1^2 |(a(x)y_x)_x|^2 dxdt.
$$
We claim that there exists a constant $C>0$ such that $\zeta^6 \leq C e^{-\frac{2M}{m(t)}}$. Indeed, $\zeta^6 e^{\frac{2M}{m(t)}}$ is a continuous function on $[0,T]$ (since $M<0$, it  converges to zero when $t \to T$). Thus, using Lemma \ref{lema para proposition} and Theorem \ref{theorem case linear}, estimate \eqref{des Proposition 6}, we have
\begin{equation*}
\begin{split}
\int_Q \rho_2^2 \left(\int_0^1 \bar y\,dx \right)^2 |(a(x)y_x)_x|^2 dxdt &\leq C \sup_{t\in [0,T]} \left\{ e^{-\frac{2M}{m(t)}} \left(\int_0^1 \bar y\,dx \right)^2  \right\} \int_Q \rho_1^2 |(a(x)y_x)_x|^2 dxdt \\
&\leq C \|(\bar y,\bar p^1,\bar p^2, \bar h)\|^{2}_\mathcal{Y} \|(y, p^1, p^2, h)\|^{2}_\mathcal{Y}.    
\end{split}    
\end{equation*}
In view of Lemma \ref{lema para proposition} and \eqref{des Proposition 6}, similar estimates also hold for 
$\int_Q \rho_2^2 \left(\int_0^1 \bar p^i\,dx \right)^2 |(a(x) p^i_x)_x|^2 dxdt$ for $i=\lbrace1,2\rbrace$. This finishes the proof of \eqref{eq:coro_2_est1}.

To prove \eqref{eq:coro_2_est2}, analogously we have
\begin{equation*}
    \begin{array}{l}
\displaystyle\int_{Q}\rho_{2}^{2}\left(\displaystyle\int_{0}^{1}{\bar y}\,dx \right)^{2} \Vert \sqrt{a(x)}y_{x}\Vert_{L^2(0,1)}^{2}dxdt \\
\leq \displaystyle\sup_{t\in [0,T]} 
\left\{ e^{-\frac{2M}{m(t)}} \left(\int_0^1 \bar y\,dx \right)^2 \right\} \int_0^T \int_0^1 \rho_1^2  \Vert \sqrt{a(x)}y_{x}\Vert_{L^2(0,1)}^{2}dxdt.
    \end{array}
\end{equation*}
Thus, by Lemma \ref{lema para proposition}, \eqref{eq:compara_rhos} and \eqref{des Proposition 5} 
we get
\begin{equation}
\label{eq:estimativa_anexo}
\begin{split}
\int_{Q}\rho_{2}^{2}\left(\displaystyle\int_{0}^{1}{\bar y}\,dx \right)^{2} \Vert \sqrt{a(x)}y_{x}\Vert_{L^2(0,1)}^{2} dxdt&\leq C \|(\bar y,\bar p^1,\bar p^2, \bar h)\|^{2}_\mathcal{Y} \int_0^T \rho_1^2 \int_0^1 a(x) |y_x|^2 dx dt\\
&\leq C \|(\bar y,\bar p^1,\bar p^2, \bar h)\|^{2}_\mathcal{Y} \int_0^T \int_0^1 \hat\rho^2 a(x) |y_x|^2 dx dt\\ 
&\leq C \|(\bar y,\bar p^1,\bar p^2, \bar h)\|^{2}_\mathcal{Y} \|(y, p^1, p^2, h)\|^{2}_\mathcal{Y}.
\end{split}    
\end{equation}

We proceed analogously to bound $\int_{Q}\rho_{2}^{2}\left(\int_{0}^{1}{\bar y}\, dx\right)^{2} \Vert \sqrt{a(x)}p^i_x \Vert_{L^2(0,1)}^{2}dxdt$, for $i=1,2$.
\end{proof}
\renewcommand{\theequation}{\thesection.\arabic{equation}}
\section{- Continuity of the map $\mathcal{A}$ }\label{appendix B}
Let be given $(y,p^{1},p^{2},h)\in \mathcal{Y}$  and consider a sequence  $((y_{n},p^{1}_{n},p^{2}_{n},h_{n}))_{n=0}^{\infty}$ which converges to  $(y,p^{1},p^{2},h)$ in $\mathcal{Y}$. For each $(\bar{y},\bar{p}^{1},\bar{p}^{2},\bar{h})\in B_{r}(0)$,
\begin{eqnarray*}
D\mathcal{A}_0(y,p^{1},p^{2},h)\cdot(\bar y, \bar p^1,\bar p^2,\bar h) &=& \bar y_t - \ell'\left(\int_0^1 y\, dx\right)\left( \int_0^1 \bar y \,dx\right) (a(x)y_x)_x\\ 
&\quad&- \ell\left(\int_0^1 y \,dx\right) (a(x) \bar y_x)_x -\bar{h} \mathds{1}_{\mathcal{O}} + \frac{1}{\mu_1} \bar{p}^{1} \mathds{1}_{\mathcal{O}_{1}} + \frac{1}{\mu_2} \bar{p}^{2} \mathds{1}_{\mathcal{O}_{2}},
\end{eqnarray*}
and if $i=\lbrace 1,2\rbrace$
\begin{equation*}
\begin{array}{l}
D\mathcal{A}_i(y,p^{1},p^{2},h)\cdot(\bar y, \bar p^1,\bar p^2,\bar h)=-\bar{p}_t^i - \ell'\left(\displaystyle\int_0^1 y \, dx\right)\left(\displaystyle\int_0^1 \bar y \, dx\right) (a(x)p_x^i)_x\\
\quad -\ell\left(\displaystyle\int_0^1 y \, dx\right) (a(x) \bar p_x^i)_x
 + \ell''\left(\displaystyle\int_0^1 y \, dx\right) \displaystyle\int_0^1 \bar y \, dx \left( \displaystyle\int_0^1 a(x) y_x p_x^i\, dx \right)\\
\quad+ \ell'\left(\displaystyle\int_0^1 y \, dx\right) \displaystyle\int_0^1 a(x) \bar y_x p_x^i \, dx + \ell'\left(\displaystyle\int_0^1 y \, dx\right) \displaystyle\int_0^1 a(x) y_x \bar p_x^i \, dx
-\alpha_i \bar y \mathds{1}_{\mathcal{O}_{i,d}}.
\end{array}
\end{equation*}
Then, for each $i = \lbrace 0, 1, 2, 3\rbrace$ we have
\begin{equation}\label{convergence of D}
    \begin{array}{l}
\|(D\mathcal{A}_i(y_{n},p^{1}_{n},p^{2}_{n},h_{n})-D\mathcal{A}_i(y,p^{1},p^{2},h))\cdot(\bar y, \bar p^1,\bar p^2,\bar h)\|_{\mathcal{Z}}\to 0,\,  
\text{as}\,\, n\to\infty. 
\end{array}
\end{equation}
Indeed, for $i=0$,
\begin{equation*}
\begin{array}{l}
(D\mathcal{A}_0(y_{n},p^{1}_{n},p^{2}_{n},h_{n})-D\mathcal{A}_0(y,p^{1},p^{2},h))\cdot(\bar y, \bar p^1,\bar p^2,\bar h)\\
 \displaystyle = - \ell''\left(\int_{0}^{1} y\,dx\right)\int_{0}^{1}(y_{n}-y)\,dx
 \left(\int_{0}^{1}\bar{y} \, dx \right)(a(x)y_{n,x})_x \\
 
 \quad \displaystyle-\ell'\left(\int_{0}^{1} y\,dx\right) \left(\int_{0}^{1}\bar{y} \, dx \right) (a(x)(y_{n,x}-y_{x}))_{x}\\ 
\quad \displaystyle -\ell'\left(\int_{0}^{1} y\,dx\right)\int_{0}^{1}(y_{n}-y)\,dx\, (a(x)\bar{y}_x)_x=X^{1}_{1}+X^{1}_{2}+X^{1}_{3}. 
\end{array}
\end{equation*}
For $X^{1}_{1}$, by Lemma \ref{lema para proposition} and Proposition \ref{Corollary 2}, we have 
\begin{eqnarray*}
\int_{0}^{T}\int_{0}^{1}\rho^{2}_{2}|X_{1}^{1}|^{2} dxdt
&\leq& C\int_{0}^{T}\int_{0}^{1}\rho_{2}^{2}
\left(\int_{0}^{1}(y_{n}-y)\,dx\right)^2
 \left(\int_{0}^{1}\bar{y} \, dx \right)^2 |(a(x)y_{n,x})_x|^2 dxdt\\
 &\leq&C\left(\int_{0}^{T}\int_{0}^{1}\rho_{2}^{2}\left(\int_{0}^{1}(y_{n}-y)\,dx\right)^2\left(\int_{0}^{1}\bar{y} \, dx \right)^2 |(a(x)(y_{n,x}-y_{x}))_x|^2 dxdt\right.\\ &\quad&+\left.\int_{0}^{T}\int_{0}^{1}\rho_{2}^{2}\left(\int_{0}^{1}(y_{n}-y)\,dx\right)^2\left(\int_{0}^{1}\bar{y} \, dx \right)^2 |(a(x)y_{x})_x|^2 dxdt\right)\\
 &\leq& C\sup_{t\in[0,T]}\left\{ e^{\frac{-2M}{m(t)}}\left(\int_{0}^{1}(y_{n}-y)\,dx\right)^{2}\right\}\left(\int_{0}^{T}\int_{0}^{1}\rho_{1}^{2}\left(\int_{0}^{1}\bar{y} \, dx \right)^2 \right.\\ &\quad&\left.|(a(x)(y_{n,x}-y_{x}))_x|^2 dxdt+\int_{0}^{T}\int_{0}^{1}\rho_{1}^{2}\left(\int_{0}^{1}\bar{y} \, dx \right)^2 |(a(x)y_{x})_x|^2 dxdt\right)\\
 &\leq&C \|(y_{n}-y),(p^{1}_{n}-p^{1}),(p^{2}_{n}-p^{2}),(h_{n}-h)\|^{2}_{\mathcal{Y}}\,\|(\bar y, \bar p^{1},\bar p^{2}, \bar h)\|^{2}_{\mathcal{Y}}\left\lbrace\|y,p^{1},p^{2},h\|^{2}_{\mathcal{Y}}\right.\\ &\quad&+\|(y_{n}-y),(p^{1}_{n}-p^{1}),(p^{2}_{n}-p^{2}),(h_{n}-h)\|^{2}_{\mathcal{Y}} \left.\right\rbrace \rightarrow 0,
\end{eqnarray*}
where was it used that $\rho_{2}^{2}=\zeta^{6}\rho_{1}^{2}\leq C e^{\frac{2M}{m(t)}}$.

For $X^{1}_{2}$, by Proposition \ref{Corollary 2},
\begin{eqnarray*}
\int_{0}^{T}\int_{0}^{1}\rho^{2}_{2}|X_{2}^{1}|^{2} dxdt&\leq& C \int_{0}^{T}\int_{0}^{1}\rho_{2}^{2}
 \left(\int_{0}^{1}\bar{y} \, dx \right)^2 |(a(x)(y_{n,x}-y_{x}))_{x}|^2 dxdt \\
&\leq&C\|(y_{n}-y),(p^{1}_{n}-p^{1}),(p^{2}_{n}-p^{2}),(h_{n}-h)\|^{2}_{\mathcal{Y}}\, \|(\bar y, \bar p^{1},\bar p^{2}, \bar h)\|^{2}_{\mathcal{Y}}\rightarrow 0.
\end{eqnarray*}
For $X^{1}_{3}$, by  Lemma \ref{lema para proposition} and again by Proposition \ref{Corollary 2}, we have
\begin{equation*}
\begin{array}{l}
\displaystyle\int_{0}^{T}\int_{0}^{1}\rho^{2}_{2}|X_{3}^{1}|^{2} dxdt\leq C \int_{0}^{T}\int_{0}^{1}\rho_{2}^{2}
    \left(\int_{0}^{1}(y_{n}-y)\,dx\right)^2 |(a(x)\bar{y}_x)_x|^2 dxdt\\
  \displaystyle  \leq C\sup_{t\in[0,T]}\left\{ e^{\frac{-2M}{m(t)}}\left(\int_{0}^{1}(y_{n}-y)\,dx\right)^{2}\right\}\int_{0}^{T}\int_{0}^{1}\rho_{1}^{2} |(a(x)\bar{y}_{x})_{x}|^{2}\,dxdt\\
  \displaystyle  \leq C\|(y_{n}-y),(p^{1}_{n}-p^{1}),(p^{2}_{n}-p^{2}),(h_{n}-h)\|^{2}_{\mathcal{Y}}\|y,p^{1},p^{2},h)\|^{2}_{\mathcal{Y}} \|(\bar y, \bar p^{1},\bar p^{2}, \bar h)\|^{2}_{\mathcal{Y}}\rightarrow 0.
    \end{array}
\end{equation*}
Now consider the term $D\mathcal{A}_{i}$, for $i=\lbrace 1, 2\rbrace$, we have
\begin{equation*}
\begin{array}{ll}
    (D\mathcal{A}_i(y_{n},p^{1}_{n},p^{2}_{n},h_{n})-D\mathcal{A}_i(y,p^{1},p^{2},h))\cdot(\bar y, \bar p^1,\bar p^2,\bar h)= X_{1}^{2}+X_{2}^{2}+X_{3}^{2}+X_{4}^{2}+X_{5}^{2},
    \end{array}
\end{equation*}
where
\begin{eqnarray*}
X^{2}_{1}&=&\ell''\left(\int_{0}^{1} y\,dx\right)\int_{0}^{1}(y_{n}-y)\, dx\int_{0}^{1}\bar{y}\, dx (a(x)p_{n,x}^{i})_{x}\\
&\quad&+\ell'\left(\int_{0}^{1} y\,dx\right)\int_{0}^{1}\bar{y}\, dx\ (a(x)(p_{n,x}^{i}-p_{x}^{i}))_{x}=X^{2}_{11}+X^{2}_{12} \\
X^{2}_{2}&=& \ell'\left(\int_{0}^{1} y\,dx\right)\int_{0}^{1}(y_{n}-y) dx\ (a(x)\bar{p}^{i}_{x})_{x}=X^{2}_{21}\\   
X^{2}_{3}&=&\left(\ell''\left(\int_{0}^{1} y_{n}\,dx\right)-\ell''\left(\int_{0}^{1} y\,dx\right)\right)\int_{0}^{1}\bar{y}\, dx\displaystyle\int_{0}^{1}a(x)y_{n,x}p_{n,x}^{i}\, dx\\
&\quad&+ \ell''\left(\int_{0}^{1} y\,dx\right)\int_{0}^{1}\bar{y}\, dx\displaystyle\int_{0}^{1}a(x)(y_{n,x}-y_{x})p_{n,x}^{i}\, dx\\ 
&\quad& + \ell''\left(\int_{0}^{1} y\,dx\right) \int_{0}^{1}\bar{y}\, dx\int_{0}^{1}a(x)(p_{n,x}^{i}-p_{x}^{i})y_{n,x}\, dx=X^{2}_{31}+X^{2}_{32}+X^{2}_{33};\\
X^{2}_{4}&=&\ell''\left(\int_{0}^{1} y\,dx\right)\int_{0}^{1}(y_{n}-y)\,dx\int_{0}^{1}a(x)\bar{y}_{x}p_{n,x}^{i}\ dx\\
&\quad&+ \ell'\left(\int_{0}^{1} y\,dx\right)\int_{0}^{1}a(x)\bar{y}_{x}(p_{n,x}^{i}-p_{x}^{i})\, dx = X^{2}_{41}+X^{2}_{42};\\
X^{2}_{5}&=&\ell''\left(\int_{0}^{1} y\,dx\right)\int_{0}^{1}(y_{n}-y)\,dx\int_{0}^{1}a(x)y_{n,x}\bar{p}_{x}^{i}\, dx\\
&\quad&+\ell'\left(\int_{0}^{1} y\,dx\right)\int_{0}^{1} a(x)(y_{n,x}-y_{x})\bar{p}_{x}^{i}\, dx=X^{2}_{51}+X^{2}_{52}.
\end{eqnarray*} 
Then, 
\begin{eqnarray*}
(D\mathcal{A}_i(y_{n},p^{1}_{n},p^{2}_{n},h_{n})-D\mathcal{A}_i(y,p^{1},p^{2},h))\cdot(\bar y, \bar p^1,\bar p^2,\bar h)
&=&X^{2}_{11}+X^{2}_{12}+X^{2}_{21}+X^{2}_{31}+X^{2}_{32}\\
&\quad&+ X^{2}_{33}+X^{2}_{41}+X^{2}_{42}+X^{2}_{51}+X^{2}_{52}.
\end{eqnarray*}
We now proceed to bound each term on the right-hand side of the previous equality. 
To this end, making use, when necessary, of Proposition \ref{Corollary 2}, Lemma  \ref{lema para proposition}, the condition $\ell\in C^{2}(\R)$, and proceeding with similar arguments given before,  we have:

\begin{equation*}
\begin{array}{l}
\displaystyle\bullet \, \int_{0}^{T}\int_{0}^{1}\rho_{2}^{2}|X_{11}^{2}|^{2}dxdt
 \leq  C\int_{0}^{T}\int_{0}^{1}\rho^{2}_{2}\left(\int_{0}^{1}(y_{n}-y)\,dx\right)^2\left(\int_{0}^{1}\bar{y}\, dx\right)^{2}|(a(x)p_{n,x}^{i})_{x}|^{2}dxdt\\
 \displaystyle\leq C\sup_{t\in[0,T]}\left\{ e^{\frac{-2M}{m(t)}}\left(\int_{0}^{1}(y_{n}-y)\,dx\right)^{2}\right\}\int_{0}^{T}\int_{0}^{1}\rho_{1}^{2}\left(\int_{0}^{1}\bar{y} \, dx \right)^2
|(a(x)p_{n,x}^{i})_{x}|^{2}dxdt\\
 \leq C \|(y_{n}-y),(p^{1}_{n}-p^{1}),(p^{2}_{n}-p^{2}),(h_{n}-h)\|^{2}_{\mathcal{Y}}\,\|(\bar y, \bar p^{1},\bar p^{2}, \bar h)\|^{2}_{\mathcal{Y}}\left\lbrace{\|y_n,p^{1}_n,p^{2}_n,h_n\|^{2}_{\mathcal{Y}}}\right.\\  \displaystyle\quad +\|(y_{n}-y),(p^{1}_{n}-p^{1}),(p^{2}_{n}-p^{2}),(h_{n}-h)\|^{2}_{\mathcal{Y}} \left.\right\rbrace \rightarrow 0.\\ \\
\displaystyle\bullet\, \int_{0}^{T}\int_{0}^{1}\rho_{2}^{2}|X_{31}^{2}|^{2}dxdt
 \leq C\int_{0}^{T}\int_{0}^{1}\rho_{2}^{2}\left|\ell''\left(\int_{0}^{1} y_{n}\,dx\right)-\ell''\left(\int_{0}^{1} y\,dx\right)\right|^{2}\left(\int_{0}^{1}\bar{y}\, dx\right)^{2}\\ 
\displaystyle \quad \left(\int_{0}^{1}a(x)y_{n,x}p_{n,x}^{i}\,dx\right)^{2}dxdt\\
\displaystyle \leq C\int_{0}^{T}\int_{0}^{1}\rho_{2}^{2}\left|\ell''\left(\int_{0}^{1} y_{n}\,dx\right)-\ell''\left(\int_{0}^{1} y\,dx\right)\right|^{2}\left(\int_{0}^{1}\bar{y}\, dx\right)^{2}\\
\displaystyle \quad \left(\|\sqrt{a(x)}y_{n,x}\|^{2}_{L^{2}(0,1)}+\|\sqrt{a(x)}p_{n,x}^{i}\|^{2}_{L^{2}(0,1)}\right)\, dxdt\\
 \displaystyle\leq C\left(\int_{0}^{T}\int_{0}^{1}\rho_{2}^{2}\left|\ell''\left(\int_{0}^{1} y_{n}\,dx\right)-\ell''\left(\int_{0}^{1} y\,dx\right)\right|^{2}\left(\int_{0}^{1}\bar{y}\, dx\right)^{2}\right.
  \|\sqrt{a(x)}y_{n,x}\|^{2}_{L^{2}(0,1)}\ dxdt\\
 \displaystyle\quad +\int_{0}^{T}\int_{0}^{1}\rho_{2}^{2}\left|\ell''\left(\int_{0}^{1} y_{n}\,dx\right)-\ell''\left(\int_{0}^{1} y\,dx\right)\right|^{2}\left(\int_{0}^{1}\bar{y}\, dx\right)^{2}\left.\|\sqrt{a(x)}p_{n,x}^{i}\|^{2}_{L^{2}(0,1)}\ dxdt\right)\\
 \leq C\|(y_{n}-y),(p^{1}_{n}-p^{1}),(p^{2}_{n}-p^{2}),(h_{n}-h)\|^{2}_{\mathcal{Y}}\|y_n,p^{1}_n,p^{2}_n,h_n)\|^{2}_{\mathcal{Y}}\|(\bar y, \bar p^{1},\bar p^{2}, \bar h)\|^{2}_{\mathcal{Y}}\rightarrow 0.\\ \\
\displaystyle\bullet\, \int_{0}^{T}\int_{0}^{1}\rho_{2}^{2}|X_{32}^{2}|^{2}dxdt 
 \leq  C	\int_{0}^{T}\int_{0}^{1}\rho_{2}^{2}\left|\int_{0}^{1}\bar{y}\, dx\right|^{2}\left(\int_{0}^{1} a(x)(y_{n,x}-y_{x})p_{n,x}^{i}\, dx\right)^{2} dxdt\\
\displaystyle\leq  C\int_{0}^{T}\int_{0}^{1}\rho_{2}^{2}\left|\int_{0}^{1}\bar{y}\, dx\right|^{2}\left(\|\sqrt{a(x)}(y_{n,x}-y_{x})\|^{2}_{L^{2}(0,1)}+\|\sqrt{a(x)}p_{n,x}^{i}\|^{2}_{L^{2}(0,1)}\right) dxdt\\
 \displaystyle\leq C\left(\int_{0}^{T}\int_{0}^{1}\rho_{2}^{2}\left|\int_{0}^{1}\bar{y}\, dx\right|^{2}\|\sqrt{a(x)}(y_{n,x}-y_{x})\|^{2}_{L^{2}(0,1)} dxdt\right.\\
 \displaystyle\quad +\left.\int_{0}^{T}\int_{0}^{1}\rho_{2}^{2}\left|\int_{0}^{1}\bar{y}\, dx\right|^{2}\|\sqrt{a(x)}p_{n,x}^{i}\|_{L^{2}(0,1)}\ dxdt\right)\\
 \leq C\|(y_{n}-y),(p^{1}_{n}-p^{1}),(p^{2}_{n}-p^{2}),(h_{n}-h)\|^{2}_{\mathcal{Y}}\|y_n,p^{1}_n,p^{2}_n,h_n)\|^{2}_{\mathcal{Y}}\|(\bar y, \bar p^{1},\bar p^{2}, \bar h)\|^{2}_{\mathcal{Y}}\rightarrow 0.\\ \\

\end{array}
\end{equation*}
\begin{equation*}
\begin{array}{l}
\displaystyle\bullet\, \int_{0}^{T}\int_{0}^{1}\rho_{2}^{2}|X_{41}^{2}|^{2}dxdt  \leq C\int_{0}^{T}\int_{0}^{1}\rho_{2}^{2}\left(\int_{0}^{1}(y_{n}-y)\,dx\right)^{2}\left(\int_{0}^{1}a(x)\bar{y}_{x}p_{n,x}^{i}\,dxdt\right)^{2}\\
 \displaystyle\leq C\sup_{t\in[0,T]}\left\{ e^{\frac{-2M}{m(t)}}\left(\int_{0}^{1}(y_{n}-y)\,dx\right)^{2}\right\}\int_{0}^{T}\int_{0}^{1}\rho_{1}^{2}\left(\|\sqrt{a(x)}\bar{y}_{x}\|^{2}_{L^{2}(0,1)}\right. +\left.\|\sqrt{a(x)}p^{i}_{n,x}\|^{2}_{L^{2}(0,1)}\right) dxdt\\
 \displaystyle\leq C\|(y_{n}-y),(p^{1}_{n}-p^{1}),(p^{2}_{n}-p^{2}),(h_{n}-h)\|^{2}_{\mathcal{Y}}\left(\|(\bar y, \bar p^{1},\bar p^{2}, \bar h)\|^{2}_{\mathcal{Y}} +  \|(y_n, p^{1}_n, p^{2}_n, h_n\|^{2}_{\mathcal{Y}}\right)\rightarrow 0. \\
\\
\displaystyle\bullet\, \int_{0}^{T}\int_{0}^{1}\rho_{2}^{2}|X_{42}^{2}|^{2}dxdt
 \leq C\int_{0}^{T}\int_{0}^{1}\rho_{2}^{2}\left|\int_{0}^{1}a(x)\bar{y}_{x}(p_{n,x}^{i}-p_{x}^{i})\, dx\right|^{2} dxdt\\
\displaystyle\leq  C \int_{0}^{T}\int_{0}^{1}\rho^{2}_{2}\left(\|\sqrt{a(x)}\bar{y}_{x}\|^{2}_{L^{2}(0,1)}+\|\sqrt{a(x)}(p^{i}_{n,x}-p^{i}_{x})\|^{2}_{L^{2}(0,1)}\right) dxdt \\

\displaystyle\leq C\|(y_{n}-y),(p^{1}_{n}-p^{1}),(p^{2}_{n}-p^{2}),(h_{n}-h)\|^{2}_{\mathcal{Y}}\,\|(\bar y, \bar p^{1},\bar p^{2}, \bar h)\|^{2}_{\mathcal{Y}}\rightarrow 0. 
\end{array}
\end{equation*}
And finally, 
\begin{equation*}
\begin{array}{l}
\displaystyle\bullet \quad \int_{0}^{T}\int_{0}^{1}\rho_{2}^{2}|X_{52}^{2}|^{2}dxdt     
\leq C\int_{0}^{T}\int_{0}^{1}\rho_{2}^{2}\left(\int_{0}^{1}(y_{n}-y)\,dx\right)^{2}\left|\int_{0}^{1} a(x)(y_{n,x}-y_{x})\bar{p}_{x}^{i}\,dx\right|^{2}dxdt\\
\displaystyle\leq C\sup_{t\in[0,T]}\left\{ e^{\frac{-2M}{m(t)}}\left(\int_{0}^{1}(y_{n}-y)\,dx\right)^{2}\right\}\int_{0}^{T}\int_{0}^{1}\rho_{1}^{2}\left(\|\sqrt{a(x)}(y_{n,x}-y_{x})\|^{2}_{L^{2}}\right.
+\left.\|\sqrt{a(x)}\bar{p}_{x}^{i}\|^{2}_{L^{2}}\right) dxdt\\
\leq C\|(y_{n}-y),(p^{1}_{n}-p^{1}),(p^{2}_{n}-p^{2}),(h_{n}-h)\|^{2}_{\mathcal{Y}}\\\left( \|(y_{n}-y),(p^{1}_{n}-p^{1}),(p^{2}_{n}-p^{2}),(h_{n}-h)\|^{2}_{\mathcal{Y}} + \|(\bar y, \bar p^{1},\bar p^{2}, \bar h)\|^{2}_{\mathcal{Y}}\right)\rightarrow 0.
\end{array}
\end{equation*}

The terms $X_{12}^{2}$ and $X_{21}^{2}$ follow similarly to the terms $X_{1}^{2}$ and $X_{1}^{3}$, respectively. Furthermore, we can proceed for $X_{33}^{2}$ in the same way as for $X_{32}^{2}$, and for $X_{51}^{2}$ analogously to $X_{41}^{2}$.

The case $i=3$ is immediate and, consequently, \eqref{convergence of D} holds. Therefore, $(y,p^{1},p^{2},h)\longmapsto\mathcal{A}^{\prime}(y,p^{1},p^{2},h)$ is continuous from $\mathcal{Y}$ into $\mathcal{L}(\mathcal{Y},\mathcal{Z})$ and as consequently, in view of classical results, we will
have that $\mathcal{A}$ is Fr\'echet-differentiable and $C^{1}$.

\renewcommand{\theequation}{\thesection.\arabic{equation}}
\section{- Well-Posedness of system \ref{eq:optimality_system}}\label{appendix C}

By applying the change of variables $\varphi^{i}(t) = p^{i}(T - t)$, $\tilde{y}(t) = y(T - t)$, and $\tilde{y}_{i,d} = y_{i,d}(T - t)$ to the second equation of the optimality system~\eqref{eq:optimality_system}, for $i = 1, 2$, we obtain the following system. For simplicity, we retain the same notation for $y$ and $y_{i,d}$ in place of $\tilde{y}$ and $\tilde{y}_{i,d}$, respectively.

\begin{equation*}
\left\{\begin{aligned}
&y_t - \left(a(x)\ell\left(\int_0^1 y dx\right) y_x\right)_x = h \mathds{1}_{\mathcal{O}} - \sum_{i=1}^{2}\frac{1}{\mu_i} \varphi^i \mathds{1}_{\mathcal{O}_i} &&\text{in}&& Q, \\
&\varphi_t^i - \left( a(x) \ell\left(\int_0^1 y dx\right) \varphi_x^i \right)_x  +  \ell'\left(\int_0^1 y dx\right) \int_0^1 ( a(x) y_x \ \varphi_x^i) dx = \alpha_i (y-y_{i,d}) \mathds{1}_{\mathcal{O}_{i,d}}  &&\text{in}&& Q,\\
&y(0,t)=y(1,t)=0, \ \ \varphi^i(0,t) = \varphi^i(1,t) = 0 &&\text{on}&& (0,T), \\
&y(\cdot,0) = y_{0}, \ \ \varphi^i(\cdot,0) = \varphi^{i}_{0} &&\text{in}&& (0,1).
\end{aligned}
\right.
\end{equation*}
Let $(w_{i})_{i}^{\infty}$ be an orthonormal basis of $H^{1}_{a}(0,1)$ such that 
\begin{equation*}
-(a(x)w_{i,x})_{x}=\lambda_{i}w_{i}.    
\end{equation*}
Fix $m\in\mathbb{N}^{*}$. Due to Carath\'eodory's theorem, there exist absolutely continuous functions $g_{im}=g_{im}(t)$ and $h_{im}=h_{im}(t)$ with $i\in\{1,2,...,m\}$ such that 
\begin{equation*}
t\in[0,T]\mapsto y_{m}(t)=\sum_{i=1}^{m}g_{im}(t)w_{i} \in H_{a}^{1}(0,1),    
\end{equation*}
and 
\begin{equation*}
t\in[0,T]\mapsto \varphi_{m}(t)=\sum_{i=1}^{m}h_{im}(t)w_{i} \in H_{a}^{1}(0,1),    
\end{equation*}
satisfy
\begin{equation}
\label{eq:galerkin_system}
\left\{\begin{aligned}
&(y_{m,t},w) - \ell\left(\int_0^1 y_{m} dx\right)((a(x) y_{m,x})_x,w) = (h1_{\mathcal{O}},w) - \sum_{i=1}^{2}\frac{1}{\mu_i} (\varphi^{i}_{m} 1_{\mathcal{O}_{i}},w) &&\text{in}&& Q, \\
&(\varphi_{m,t}^i,\hat{w}) - \ell\left(\int_0^1 y_{m} dx\right) ((a(x) \varphi_{mx}^i)_x,\hat{w})\\  
&+ \ell'\left(\int_0^1 y_{m} dx\right)\left(\int_0^1 a(x) y_{m,x} \ \varphi_{m,x}^idx, \hat{w}\right) = (\alpha_i (y_m -y_{i,d})1_{\mathcal{O}_{i,d}},\hat{w})   &&\text{in}&& Q,\\
&y_{m}(0,t)=y_{m}(1,t)=0, \ \ \varphi_{m}^i(0,t) = \varphi_{m}^i(1,t) = 0 &&\text{on}&& (0,T), \\
&y_{m}(\cdot,0)\to y_{0}, \ \ \varphi_{m}^i(\cdot,0)\to \varphi^{i}_{0}   &&\text{in}&& (0,1),
\end{aligned}
\right.
\end{equation}
for any $w,\hat{w}\in [w_{1},w_{2},...,w_{m}]$, where $(\cdot,\cdot)=(\cdot,\cdot)_{L^{2}(0,1)}$ and for convenience, we will adopt $\| \cdot \| = \| \cdot \|_{L^{2}(0,1)}$.\\  

\noindent\textbf{Estimate I}. 
Taking $w=y_{m}$ and $\hat{w}=\varphi^{i}_{m}$ in \eqref{eq:galerkin_system}, then 
\begin{eqnarray*}
\frac{1}{2}\frac{d}{dt}\left(\|y_{m}\|^{2}+\sum_{i=1}^{2}\|\varphi_{m}^{i}\|^{2}\right)&+& \ell\left(\int_0^1 y_{m} dx\right)\left(\|\sqrt{a}y_{m,x}\|^{2}+\sum_{i=1}^{2}\|\varphi^{i}_{m,x}\|\right)\\&+& \ell'\left(\int_{0}^{1}y_m dx\right)\sum_{i=1}^{2}\int_{0}^{1}\varphi^{i}_{m}(\sqrt{a}y_{m,x},\sqrt{a}\varphi^{i}_{m,x})\\&=&(h,y_{m})- \sum_{i=1}^{2}\frac{1}{\mu_{i}}(y_{m},\varphi^{i}_{m})+\sum_{i=1}^{2}(y_{m},\varphi^{i}_{m})-\sum_{i=1}^{2}(y_{i,d},\varphi^{i}_{m}).
\end{eqnarray*}
Note that, by the fact that $\ell\in C^{2}(\mathbb{R})$ there exists constants $C_{0}$ and $C_{1}$ such that, 
\begin{eqnarray}
\label{eq:bounds_on_l}
\left|\ell\left(\int_{0}^{1}y_{m}dx\right)\right|\leq C_{0} \ \ \text{and} \ \   \left|\ell'\left(\int_{0}^{1}y_{m}dx\right)\right|\leq C_{1}.
\end{eqnarray}
Taking $C=max\{C_{0},C_{1}\}$ and denoting by $C_{*}$ a generic constant, applying the Cauchy-Schwarz inequality, we have 
\begin{equation*}
\begin{array}{l}
\dfrac{1}{2}\dfrac{d}{dt}\left(\|y_{m}\|^{2}+\displaystyle\sum_{i=1}^{2}\|\varphi_{m}^{i}\|^{2}\right)+C\left(\|\sqrt{a}y_{m,x}\|^{2}+\displaystyle\sum_{i=1}^{2}\|\sqrt{a}\varphi_{m,x}^{i}\|^{2}\right)\vspace{0.1cm}\\
-\,C\displaystyle\sum_{i=1}^{2}\|\varphi^{i}_{m}\|\|\sqrt{a}\varphi_{m,x}^{i}\|\|\sqrt{a}y_{m,x}\| \leq C_{*}\left(\|h\|\|y_{m}\|+\displaystyle\sum_{i=1}^{2}\|y_{m}\|\|\varphi_{m}^{i}\|+\sum_{i=1}^{2}\|y_{id}\|\|\varphi_{m}^{i}\|\right).
\end{array}
\end{equation*}
Using the Young's inequality in the right side we obtain, for each $i=1,2.$, that
\begin{equation}\label{systemwellposedness}
\begin{split}
\frac{1}{2}\frac{d}{dt}\left(\|y_{m}\|^{2}+\sum_{i=1}^{2}\|\varphi_{m}^{i}\|^{2}\right)&+C\left(1-\sum_{i=1}^{2}\|\varphi^{i}_{m}\|^{2}\right)\|\sqrt{a}y_{m,x}\|^{2}+\frac{3C}{4}\sum_{i=1}^{2}\|\sqrt{a}\varphi_{m,x}^{i}\|^{2}\\
&\leq C_{*}\left(\|h\|^{2}+\|y_{m}\|^{2}+\sum_{i=1}^{2}\|\varphi_{m}^{i}\|^{2}+\sum_{i=1}^{2}\|y_{id}\|^{2}\right). 
\end{split}
\end{equation} 
Define $\displaystyle\bar{A}(t)=\sum_{i=1}^{2}\|\varphi^{i}_{m}\|^{2}$, for all $t\in[0,t_{m}]$. We will proof that $\bar{A}(t)<\frac{1}{2}$. Indeed, assuming by contradiction that is not true, then there exist $t^{*}_{m}$ such that $\bar{A}(t)<\frac{1}{2}$ and $\bar{A}(t^{*}_{m})=\frac{1}{2}$. Therefore, we have 
\begin{equation*}
\|\varphi_{0}^{1}\|^{2}+\|\varphi_{0}^{2}\|^{2}<\frac{1}{2}.    
\end{equation*}
By hypothesis, there exist $\sigma_{0}>0$ such that 
\begin{equation}\label{smalldatawellposedness}
\|h\|^{2}_{L^{2}((0,T);L^2(\mathcal{O}))}+\|y_{0}\|_{H^{1}_{a}(0,1)}+\sum_{i=1}^{2}\|y_{i,d}\|^{2}_{L^{2}((0,T);L^2(\mathcal{O}_{i}))}+\sum_{i=1}^{2}\|\varphi_{0}^{i}\|_{L^{2}(0,1)}^{2}<\sigma_{0}.
\end{equation}
Therefore, we have 
\begin{equation}\label{boundedforsmall}
\|h\|^{2}_{L^{2}((0,T);L^2(\mathcal{O}))}+\|y_{0}\|_{H^{1}_{a}(0,1)}+\sum_{i=1}^{2}\|y_{i,d}\|^{2}_{L^{2}((0,T);L^2(\mathcal{O}_{i}))}+\sum_{i=1}^{2}\|\varphi_{0}^{i}\|_{L^{2}(0,1)}^{2}<\frac{1}{C_{*}8}.
\end{equation}
Integration \eqref{systemwellposedness} from  to $0$ from $t_{m}^{*}$ and using Gronwall's inequality we deduce that, 
\begin{equation*}
\begin{split}
&\frac{1}{2}\left(\|y_{m}(t_{m}^{*})\|^{2}+\sum_{i=1}^{2}\|\varphi_{m}^{i}(t_{m}^{*})\|^{2}\right)+C\int_{0}^{t^{*}_{m}}\|\sqrt{a}y_{m,x}\|^{2}ds+\frac{3C}{4}\int_{0}^{t^{*}_{m}}\sum_{i=1}^{2}\|\sqrt{a}\varphi_{m,x}^{i}\|^{2}ds\\
&\leq C_{*}\left(\|h\|^{2}_{L^{2}((0,T),L^2(O))}+\|y_{0}\|_{H^{1}_{a}(0,1)}+\sum_{i=1}^{2}\left(\|y_{i,d}\|^{2}_{L^{2}((0,T),L^2(O_{i}))}+\|\varphi_{0}^{i}\|_{L^{2}(0,1)}^{2}\right)\right.\\ 
&+\left .\int_{0}^{t^{*}_{m}}\left(\|y_{m}\|^{2}+\sum_{i=1}^{2}\|\varphi_{m}^{i}\|^{2}\right)
ds\right)\\
&\leq C_{*}\left(\|h\|^{2}_{L^{2}((0,T);L^2(\mathcal{O}))}+\|y_{0}\|_{H^{1}_{a}(0,1)}+\sum_{i=1}^{2}\|y_{i,d}\|^{2}_{L^{2}((0,T);L^2(\mathcal{O}_{i}))}+\sum_{i=1}^{2}\|\varphi_{0}^{i}\|^{2}_{L^{2}(0,1)}\right).
\end{split}
\end{equation*}
By \eqref{boundedforsmall}, we can conclude that
\begin{equation*}
\bar{A}(t_{m}^{*})=\sum_{i=1}^{2}\|\varphi^{i}_{m}(t_{m}^{*})\|^{2}<\frac{1}{2},    
\end{equation*}
but this is a contradiction. Then, 
\begin{equation}\label{energyestimatewellpo1}
\begin{split}
&\|y_{m}(t)\|^{2}+\sum_{i=1}^{2}\|\varphi_{m}^{i}(t)\|^{2}+C\int_{0}^{t}\|\sqrt{a}y_{m,x}\|^{2}ds+\frac{3C}{2}\int_{0}^{t}\sum_{i=1}^{2}\|\sqrt{a}\varphi_{m,x}^{i}\|^{2}ds\\ 
&\leq C_{*}\left(\|h\|^{2}_{L^{2}((0,T);L^2(\mathcal{O}))}+\|y_{0}\|_{H^{1}_{a}(0,1)}+\sum_{i=1}^{2}\|y_{i,d}\|^{2}_{L^{2}((0,T);L^2(\mathcal{O}_{i}))}+\sum_{i=1}^{2}\|\varphi_{0}^{i}\|^{2}_{L^{2}(0,1)}\right).
\end{split}
\end{equation}
As the term on right side of \eqref{energyestimatewellpo1} is independent of m, we can extend the solution $(y_{m},\varphi^{i}_{m})$ to the entire interval $[0,T]$ and in same way we can estimate \eqref{energyestimatewellpo1} for $t\in[0,T]$,
\begin{equation*}
\begin{split}
&\|y_m\|_{L^{\infty}(0,T;L^{2}(0,1))}^{2}+\sum_{i=1}^{2}\|\varphi_{m}^{i}\|_{L^{\infty}(0,T;L^{2}(0,1))}^{2}+\|\sqrt{a}y_{m,x}\|_{L^{2}(0,T;L^{2}(0,1))}^{2}+\sum_{i=1}^{2}\|\sqrt{a}\varphi_{m,x}^{i}\|_{L^{2}(0,T;L^{2}(0,1))}^{2}\\ 
&\leq C_{*}\left(\|h\|^{2}_{L^{2}((0,T);L^2(\mathcal{O}))}+\|y_{0}\|_{H^{1}_{a}(0,1)}+\sum_{i=1}^{2}\|y_{i,d}\|^{2}_{L^{2}((0,T);L^2(\mathcal{O}_{i}))}+\sum_{i=1}^{2}\|\varphi_{0}^{i}\|^{2}_{L^{2}(0,1)}\right) = \mathcal{K}_1.
\end{split}
\end{equation*}

\textbf{Estimate II}. Taking $w=y_{m,t}$ and $\hat{w}=\varphi^{i}_{m,t}$ in \eqref{eq:galerkin_system}, and using the constants $C_0$ and $C_1$ in \eqref{eq:bounds_on_l}, we get
\begin{equation*}
\begin{split}
&\|y_{m,t}\|^2 + \sum_{i=1}^2 \|\varphi^i_{m,t}\|^2 + C_0 \frac{1}{2}\frac{d}{dt} \left( \|\sqrt{a}y_{m,x}\|^2 + \sum_{i=1}^2 \|\sqrt{a}y_{m,x}\|^2 \right)\\
&\,\, - C_1 \sum_{i=1}^2 \|\sqrt{a} y_{m,x}\|\|\sqrt{a}\varphi^i_{m,x}\|\|\varphi^i_{m,t}\| 
\leq \|h\|^2 + \frac{1}{4}\|y_{m,t}\|^2 + C\sum_{i=1}^2 \|\varphi^i_m\|^2 + \frac{1}{4}\|y_{m,t}\|^2\\
& \,\, + 2 C \|y_m\|^2 + \frac{1}{4}\sum_{i=1}^2  \|\varphi^i_{m,t}\|^2 + C \sum_{i=1}^2  \|y_{i,d}\|^2 
+ \frac{1}{4}\sum_{i=1}^2  \|\varphi^i_{m,t}\|^2.
\end{split}
\end{equation*}
Using Young's inequality there exists a constant $D>0$ such that 
\begin{equation*}
\begin{split}
&\|y_{m,t}\|^2 + \sum_{i=1}^2  \|\varphi^i_{m,t}\|^2 + \frac{d}{dt} \left( \|\sqrt{a}y_{m,x}\|^2 + \sum_{i=1}^2  \|\sqrt{a}\varphi^i_{m,x}\|^2 \right) \\
&\leq D \left[ \|h\|^2 + \sum_{i=1}^2  \|y_{i,d}\|^2 + \|y_m\|^2 + \sum_{i=1}^2  \|\varphi^i_m\|^2  + \sum_{i=1}^2  \|\sqrt{a}y_{m,x}\|^2 \|\sqrt{a}\varphi^i_{m,x}\|^2 \right].  
\end{split}
\end{equation*}
Integrating in $t$ on $[0,t]$ and applying Gronwall's inequality we get
\begin{equation*}
\begin{split}
&\|y_{m,t}\|^2_{L^2(0,T;L^2(0,1))} + \sum_{i=1}^2  \|\varphi^i_{m,t}\|^2_{L^2(0,T;L^2(0,1))} + \left( \|\sqrt{a}y_{m,x}(t)\|^2 + \sum_{i=1}^2  \|\sqrt{a}\varphi^i_{m,x}(t)\|^2 \right) \\
&\leq D \left[ \|\sqrt{a}y_{m,x}(0)\|^2 + \sum_{i=1}^2  \|\sqrt{a}\varphi^i_{m,x}(0)\|^2 +\|h\|^2_{L^2(0,T;L^2(\mathcal{O}))} + \sum_{i=1}^2  \|y_{i,d}\|^2_{L^2(0,T;L^2(\mathcal{O}_{i}))}\right.\\
&\,\, + \|y_m\|^2_{L^\infty(0,T;L^2(0,1))}T  + \sum_{i=1}^2  \|\varphi^i_m\|^2_{L^\infty(0,T;L^2(0,1))}T    \\
&\left.\,\, + \int_0^t \left(  \|\sqrt{a}y_{m,x}\|^2 + \sum_{i=1}^2  \|\sqrt{a}\varphi^i_{m,x}\|^2 \right) \left(  \|\sqrt{a}y_{m,x}\|^2 + \sum_{i=1}^2  \|\sqrt{a}\varphi^i_{m,x}\|^2 \right) \right] \\
&\leq D \left[ \|\sqrt{a}y_{m,x}(0)\|^2 + \sum_{i=1}^2  \|\sqrt{a}\varphi^i_{m,x}(0)\|^2 +\|h\|^2_{L^2(0,T;L^2(\mathcal{O}))} + \sum_{i=1}^2  \|y_{i,d}\|^2_{L^2(0,T;L^2(\mathcal{O}_{i}))}  \right. \\
&\left. \,\,+ \|y_m\|^2_{L^\infty(0,T;L^2(0,1))}T + \sum_{i=1}^2  \|\varphi^i_m\|^2_{L^\infty(0,T;L^2(0,1))}T \right] e^{D \int_0^t \left(  \|\sqrt{a}y_{m,x}\|^2 + \sum_{i=1}^2  \|\sqrt{a}\varphi^i_{m,x}\|^2 \right) ds}.
\end{split}
\end{equation*}
Finally, using estimate I,
\begin{equation*}
\begin{split}
&\|y_{m,t}\|^2_{L^2(0,T,L^2(0,1))} + \sum_{i=1}^2  \|\varphi^i_{m,t}\|^2_{L^2(0,T;L^2(0,1))} + \|\sqrt{a}y_{m,x}\|^2_{L^\infty(0,T;L^2(0,1))} + \sum_{i=1}^2  \|\sqrt{a}\varphi^i_{m,x}\|^2_{L^\infty(0,T;L^2(0,1))} \\
&\leq D \left[ \|\sqrt{a}y_{m,x}(0)\|^2_{L^2(0,1)} + \sum_{i=1}^2  \|\sqrt{a}\varphi^i_{m,x}(0)\|^2_{L^2(0,1)} +\|h\|^2_{L^2(0,T;L^2(\mathcal{O}))}  \right. \\
&\left. + \sum_{i=1}^2  \|y_{i,d}\|^2_{L^2(0,T;L^2(\mathcal{O}_{i}))} + 3\mathcal{K}_1 T  \right] e^{3 D \mathcal{K}_1} =\mathcal{K}_2.
\end{split}
\end{equation*}
\textbf{Estimate III}. Taking $w=-(a(x)y_{m,x})_x$ and $\hat{w}=-(a(x) \varphi^{i}_{m,x})_x$ in \eqref{eq:galerkin_system}, using the bounds \eqref{eq:bounds_on_l} and proceeding as in the above estimates, we get
\begin{equation*}
\begin{split}
&-(y_{m,t}, (a y_{m,x})_x) - \sum_{i=1}^2  (\varphi^i_{m,t}, (a \varphi^i_{m,x})_x) + C_0\left(\| (a y_{m,x})_x \|^2 + \sum_{i=1}^2  \| (a \varphi^i_{m,x})_x \|^2 \right) \\
& - C_1 \sum_{i=1}^2  \|\sqrt{a} y_{m,x}\|\|\sqrt{a}\varphi^i_{m,x}\|\|(a\varphi^i_{m,x})_x\| \leq C_\epsilon \|h\|^2 + \epsilon\|(a y_{m,x})_x\|^2 + C_\epsilon \sum_{i=1}^2  \|\varphi^i_m\|^2  + \epsilon\|(a y_{m,x})_x\|^2 \\
& + C_\epsilon 2 \|y_m\|^2 + \epsilon\sum_{i=1}^2  \|(a\varphi^i_{m,x})_x\|^2 + C_\epsilon\sum_{i=1}^2  \|y_{i,d}\|^2 + \epsilon\sum_{i=1}^2  \|(a\varphi^i_{m,x})_x\|^2.
\end{split}
\end{equation*}
Rearranging and using Young's inequality, for some $D>0$, we have 
\begin{equation*}
\begin{split}
&\frac{d}{dt}\left( \|\sqrt{a} y_{m,x} \|^2 + \sum_{i=1}^2  \|\sqrt{a} \varphi^i_{m,x}\|^2\right) + \|(a y_{m,x})_x\|^2 + \sum_{i=1}^2  \|(a \varphi^i_{m,x})_x\|^2 \\
&\leq D \left( \|h\|^2 + \sum_{i=1}^2  \|y_{i,d}\|^2 + \|y_m\|^2 + \sum_{i=1}^2  \|\varphi^i_ m\|^2 + 
\sum_{i=1}^2 \|\sqrt{a}y_{m,x}\|^2\|\sqrt{a}\varphi^i_{m,x}\|^2 \right).
\end{split}
\end{equation*}
Integrating in $t$ on $[0,t]$, using Gronwall's inequality and proceeding as in estimate II, we have that
\begin{equation*}
\begin{split}
&\|\sqrt{a} y_{m,x}\|^2_{L^\infty(0,T;L^2(0,1))} +  \|(a y_{m,x})_x\|^2_{L^2(0,T;L^2(0,1))} + \sum_{i=1}^2  \|\sqrt{a} \varphi^i_{m,x}\|^2_{L^\infty(0,T;L^2(0,1))}  \\
& + \sum_{i=1}^2  \|(a \varphi^i_{m,x})_x\|^2_{L^2(0,T;L^2(0,1))} \leq  D\left[ \|\sqrt{a} y_{m,x}(0) \|^2 + \sum_{i=1}^2  \|\sqrt{a} \varphi^i_{m,x}(0)\|^2 + \|h\|^2_{L^2(0,t;L^2(\mathcal{O}))}\right.\\
& \left. + \sum_{i=1}^2  \|y_{i,d}\|^2_{L^2(0,t;L^2(\mathcal{O}_{i}))} + 3 \mathcal{K}_1 \right] e^{3D\mathcal{K}_1} = \mathcal{K}_3.
\end{split}
\end{equation*}
Since $\mathcal{K}_1$, $\mathcal{K}_2$ and $\mathcal{K}_3$ do not depend on $m$, the three estimates above imply that the sequences $(y_m)$, $(\varphi^1_m)$ and $(\varphi^1_m)$ are bounded in 
$$
L^2(0,T;L^2(0,1)) \cap H^1(0,T;H^2_a(0,1)).
$$
Therefore, there exist subsequences $(y_{m_j})$, $(\varphi^i_{m_j})$, $i=1,2$, such that
$$
y_{m_j} \rightharpoonup y,\qquad \varphi^i_{m_j} \rightharpoonup \varphi^i, \qquad \text{as } j\to \infty, 
$$
weakly in $L^2(0,T;L^2(0,1)) \cap H^1(0,T;H^2_a(0,1))$.
In fact, since the immersion $H^2_a(0,1)$ in $H^1_a(0,1)$ is compact, by the theorem of Aubin-Lions we get that
$$
y_{m_j} \to y,\qquad \varphi^i_{m_j} \to \varphi^i, \qquad \text{as } j\to \infty, 
$$
strongly in $H^1_a(0,1)$. Using the continuity of $\ell$ and $\ell'$,
at least for a subsequence, we can pass to the limit in $m$ in  all the terms of the the approximate system \eqref{eq:galerkin_system}, including the integral term 
$\int_0^1 a(x) y_{mx} \ \varphi_{mx}^idx$.

The uniqueness is proved by standard methods for nonlinear systems, taking into consideration the bounds established in estimates I - III for $y$, $y_x$ and $\varphi^i$ and $\varphi^i_x$, for $i=1,2$.

\section*{Acknowledgments}
\noindent The authors express their sincere thanks to the referee for his/her valuable comments. The research of João Carlos Barreira and Suerlan Silva are partially supported by Coordena\c{c}\~ao de Aperfei\c{c}oamento de Pessoal de N\'ivel Superior-Brasil (CAPES)-Finance Code 001 and the author Juan L\'imaco is also partially supported by CNPQ-Brasil.

\section*{Declarations}
\noindent Competing Interests: The authors have not disclosed any competing interests.

\bibliographystyle{abbrv}
\bibliography{bib_hierarquico}

Corresponding author: Juan L\'imaco.
\end{document}